\documentclass[12pt]{amsart}
\usepackage{geometry}
\usepackage{wasysym,amssymb,amsthm,cite,color,eufrak,enumitem,epsfig,float,indentfirst,graphicx}
\usepackage[bookmarksnumbered, colorlinks, linktocpage, plainpages]{hyperref}
\newtheorem{theorem}{Theorem}[section]

\newtheorem{lemma}[theorem]{Lemma}
\newtheorem{proposition}[theorem]{Proposition}
\theoremstyle{definition}

\newtheorem{question}{Question}

\theoremstyle{remark}
\newtheorem{remark}[theorem]{Remark}

\numberwithin{equation}{section}

\begin{document}

\title[Complex Geodesics and Complex Monge-Amp\`{e}re Equations]
{Complex geodesics and complex Monge-Amp\`{e}re equations with boundary singularity}
\date{}

\author[Xiaojun Huang]{Xiaojun Huang$^{{\rm \lowercase{a}}}$}
\thanks{$^a$Partially supported by NSF grants DMS-1665412 and DMS-2000050.}

\author[Xieping Wang]{Xieping Wang$^{{\rm \lowercase{b}}}$}
\thanks{$^b$Partially supported by NSFC grants 11771412, 12001513, NSF of Anhui Province grant 2008085QA18, and CPSF grant 2017M620072.}

\address{Department of Mathematics, Rutgers University, New Brunswick, NJ 08903, USA}
\email{huangx@math.rutgers.edu}

\address{School of Mathematical Sciences, University of Science and Technology of China, Hefei 230026, Anhui, People's Republic of China}
\email{xpwang008@ustc.edu.cn}

\begin{abstract}
We study complex geodesics and complex Monge-Amp\`{e}re equations on boun-\ ded strongly linearly convex domains in $\mathbb C^n$. More specifically, we prove the uniqueness of complex geodesics with  prescribed boundary value and direction in such a domain, when its boundary is of minimal regularity. The existence of such complex geodesics was proved by the first author in the early 1990s, but the uniqueness was left open. Based on the existence and the uniqueness proved here, as well as  other previously obtained results, we solve a homogeneous complex Monge-Amp\`{e}re equation with prescribed boundary singularity, which was first considered by Bracci et al. on smoothly bounded strongly convex domains in $\mathbb C^n$.



\end{abstract}
\maketitle

\section{Introduction}
Since the celebrated work of Bedford-Taylor \cite{Bedford76, Bedford82} and Yau \cite{Yau}, complex Monge-Amp\`{e}re equations have been an important part in the study of pluripotential theory, several complex variables and complex geometry. In this paper, we are interested in the theory of complex geodesics and its connections with homogeneous complex Monge-Amp\`{e}re equations with prescribed singularity. A first major breakthrough on this subject was made by Lempert in his famous work \cite{Lempert81}. We first prove a boundary uniqueness result for complex geodesics of a bounded strongly linearly convex domain with $C^3$-smooth boundary. Using this result as a  basic tool, we construct for such a domain a foliation by complex geodesics initiated from a fixed boundary point. Such a foliation is then used to construct a pluricomplex Poisson kernel which solves a homogeneous complex Monge-Amp\`{e}re equation with prescribed
boundary singularity. This kernel reduces to the classical Poisson kernel when the domain is the open unit disc in the complex plane.

To start with, we recall that a domain $\Omega\subset {\mathbb C}^n$ with $n>1$ is called  {\it strongly linearly convex} if it has a $C^2$-smooth boundary and admits a $C^2$-defining function $r\!:\mathbb C^n\to\mathbb R$ whose real Hessian is positive definite on the complex tangent space of $\partial\Omega$, i.e.,
$$
\sum_{j,\, k=1}^n\frac{\partial^2 r}{\partial z_j\partial\overline
z_k}(p)v_j\overline{v}_k
>\left|\sum_{j,\, k=1}^n\frac{\partial^2 r}{\partial z_j\partial z_k}(p)v_jv_k\right|
$$
for all $p\in\partial\Omega$ and  non-zero $v=(v_1,\ldots,v_n)\in T_p^{1,\, 0}\partial\Omega$; see, e.g., \cite{Andersson04, HormanderConvexity}. Strong linear convexity is a natural notion of convexity in several complex variables,  which is weaker than the usual strong convexity but stronger than strong pseudoconvexity. It is also known that there are  bounded strongly linearly convex domains with real analytic boundary, which are not biholomorphic to convex ones; see \cite{Pflug-Zwonek} and also \cite{Jarnicki-Pflugbook}.

Next, we recall  briefly the definitions of the Kobayashi-Royden metric and the Kobayashi distance; see \cite{Abatebook,Kobayashibook,Jarnicki-Pflugbook} and the references therein for a complete insight. Let $\Delta\subset\mathbb C$ be the open unit disc. The {\it Kobayashi-Royden metric} $\kappa_{\Omega}$ on a domain $\Omega\subset\mathbb C^n$ is  the pseudo-Finsler metric defined by
$$
\kappa_{\Omega}(z, v):=\inf\big\{\lambda>0\,|\;\exists \;\varphi\in \mathcal{O}(\Delta,\, \Omega)\!: \varphi(0)=z,\, \varphi'(0)=\lambda^{-1}v \big\}, \quad (z, v)\in \Omega\times\mathbb C^n,
$$
where $\mathcal{O}(\Delta,\, \Omega)$ denotes the set of holomorphic mappings from $\Delta$ to $\Omega$.
The {\it Kobayashi distance} on $\Omega$ is then defined by
$$
k_{\Omega}(z, w)=\inf_{\gamma\in \Gamma}\int_{0}^{1}\kappa_{\Omega}(\gamma(t), \,\gamma'(t))dt,
\quad (z, w)\in \Omega\times \Omega,
$$
where $\Gamma$ is the set of piecewise $C^1$-smooth curves $\gamma\!:[0, 1] \rightarrow \Omega$ such that $\gamma(0)=z$ and $\gamma(1)=w$. For the open unit disc  $\Delta\subset\mathbb C$, $k_{\Delta}$ coincides with the classical \textit{Poincar\'{e} distance}, i.e.,
$$
k_{\Delta} (\zeta_1, \zeta_2)=\tanh^{-1}\bigg|\frac{\zeta_1-\zeta_2}{1-\zeta_1\overline{\zeta}_2}\bigg|,
\quad (\zeta_1, \zeta_2)\in \Delta\times\Delta.
$$
A holomorphic mapping $\varphi\!:\Delta\to\Omega$ is called a {\it complex geodesic} of $\Omega$ in the sense of Vesentini \cite{Vesentini81}, if it is an isometry between $k_{\Delta}$ and $k_{\Omega}$, i.e.,
\begin{equation*}\label{isometry}
k_{\Omega}(\varphi(\zeta_1),\,\varphi(\zeta_2))=k_{\Delta}(\zeta_1, \zeta_2)
\end{equation*}
for all $\zeta_1$, $\zeta_2\in\Delta$.

The existence of complex geodesics with prescribed data is a very subtle problem.  In his two  important papers \cite{Lempert81, Lempert84}, Lempert addressed this problem for strongly (linearly) convex domains in $\mathbb C^n$ by a rather involved deformation argument; see also \cite{Poletsky83} for a related  work for
extremal mappings on more general pseudoconvex domains. Lempert proved that complex geodesics exist in great abundance on bounded strongly linearly convex domains and enjoy certain nice properties. To be more specific, let $\Omega\subset\mathbb C^n\, (n>1)$ be a bounded strongly linearly convex domain with $C^{m,\,\alpha}$-smooth boundary, where $m\ge 2$ and $\alpha\in(0, 1)$. Then every complex geodesic $\varphi$ of $\Omega$ is a proper holomorphic embedding of $\Delta$ into $\Omega$, and is $C^{m-1,\,\alpha}$-smooth up to the boundary. And there exists a holomorphic mapping $\varphi^{\ast}\!:\Delta\to\mathbb C^n$, also
$C^{m-1,\,\alpha}$-smooth up to the boundary, such that
$$
\varphi^{\ast}|_{\partial\Delta}(\zeta)=\zeta\mu(\zeta)\overline{\nu\circ\varphi(\zeta)},
$$
where $0<\mu\in C^{m-1,\,\alpha}(\partial\Delta)$ and $\nu$ denotes the unit outward normal vector field of $\partial\Omega$. Such a mapping $\varphi^{\ast}$ is unique up to a positive constant multiple, and can be normalized so that $\langle\varphi', \, \overline{\varphi^{\ast}}\rangle=1$ on $\overline{\Delta}$,
where $\langle\;,\,\rangle$ denotes the standard Hermitian inner product on $\mathbb C^n$, i.e.,
$$
\langle z,  w\rangle:=\sum_{j=1}^{n}z_j\overline{w}_j
$$
for $z=(z_1,\ldots,z_n)$, $w=(w_1,\ldots,w_n)\in\mathbb C^n$. The mapping $\varphi^{\ast}$ with such a normalization condition is  usually called the {\it dual mapping} of $\varphi$.  Lempert also proved, among other things, that for every $z\in \Omega$ and $v\in \mathbb C^n\setminus\{0\}$  there is a unique complex geodesic $\varphi$ of $\Omega$ such that $\varphi(0)=z$ and $\varphi'(0)=v/\kappa_{\Omega}(z, v)$. Similar to this interior existence and uniqueness result, we prove the following boundary analogue, which is the first main result of this paper:

\begin{theorem}\label{mainthm1}
Let $\Omega\subset\mathbb C^n\, (n>1)$ be a bounded strongly linearly convex domain with $C^3$-smooth boundary. Let $p\in\partial\Omega$ and $\nu_p$ be the unit outward normal to $\partial\Omega$ at $p$. Then for every $v\in\mathbb C^n\setminus T_p^{1,\,0}\partial\Omega$  with $\langle v, \nu_p\rangle>0$, there is a unique complex geodesic $\varphi$ of $\Omega$ {\rm(}up to a parabolic automorphism of $\Delta$ fixing $1${\rm)} such that $\varphi(1)=p$ and $\varphi'(1)=v$. Moreover, $\varphi$ is uniquely determined by the additional {\rm(}and always realizable{\rm)} condition that
$$
\left.\frac{d}{d\theta}\right|_{\theta=0}|\varphi^{\ast}(e^{i\theta})|=0,
$$
where $\varphi^{\ast}$ is the  dual mapping  of $\varphi$.
\end{theorem}

\begin{remark}\label{rmk:existence of geodesics}
The requirement in Theorem \ref{mainthm1} that $v\in\mathbb C^n\setminus T_p^{1,\,0}\partial\Omega$ and $\langle v, \nu_p\rangle>0$ is also a necessary condition for the existence of a complex geodesic $\varphi$ of $\Omega$ with prescribed value $p$ and derivative $v$ at $1$. The reason is the following: Since each such  $\varphi$ is proper and belongs to $C^1(\overline{\Delta})$, it follows that $\varphi(\partial\Delta)\subset\partial\Omega$ and thus $d\varphi_1(T_1\partial\Delta)\subset T_p\partial\Omega$, i.e., $iv\in T_p\partial\Omega$.  Note also that $\Omega$ is strongly pseudoconvex,  we can take a $C^2$-defining function $r$ for $\Omega$ which is strictly plurisubharmonic on some neighborhood of $\overline{\Omega}$. Then the classical Hopf lemma applied to $r\circ\varphi$ yields that $dr_p(\varphi'(1))>0$, i.e., ${\rm Re}\langle v, \nu_p\rangle>0$. Therefore, we conclude that
$\langle v, \nu_p\rangle$ is a positive number, as required.
\end{remark}

When $\Omega$ has a $C^{14}$-smooth boundary, the first part of Theorem \ref{mainthm1} was proved by Chang-Hu-Lee \cite{Chang-Hu-Lee88} by generalizing Lempert's deformation theory (see \cite{Lempert81}) to the boundary
via the Chern-Moser-Vitushkin  normal form theory. However, when $\Omega$ has only a $C^3$-smooth boundary, the situation is much more subtle. In \cite{Huang-Pisa94}, the first author established the existence part of Theorem \ref{mainthm1} for bounded strongly convex domains with $C^{3}$-smooth boundary. His proof works equally well for the strongly linearly convex case, in view of the work of Lempert \cite{Lempert84} and Chang-Hu-Lee \cite{Chang-Hu-Lee88}. In other words, the existence part of Theorem \ref{mainthm1} was
essentially known in \cite{Huang-Pisa94}. This was done by establishing a non-degeneracy property for extremal mappings (w.r.t. the Kobayashi-Royden metric) of bounded strongly pseudoconvex domains in $\mathbb C^n$ with $C^3$-smooth boundary, whose proof also indicates that the uniqueness part of Theorem \ref{mainthm1} holds for complex geodesics with direction almost tangent to $\partial\Omega$ (under the slightly stronger assumption that $\partial\Omega$ is $C^{3,\, \alpha}$-smooth); see
\cite[Lemma 3]{Huang-Pisa94} (and its proof) for details. The main contribution of this paper to Theorem \ref{mainthm1} is to provide a proof of the uniqueness part {\it in full generality}, which has been left open since \cite{Huang-Pisa94}.

Theorem \ref{mainthm1} has important applications in solving degenerate complex Monge-Amp\`{e}re equations with prescribed boundary singularity. Indeed, it can be applied to construct for every bounded strongly linearly convex domain with $C^3$-smooth boundary a foliation with complex geodesic discs (namely, the image of complex geodesics) initiated from a fixed boundary point as its holomorphic leaves, or equivalently, a so-called boundary spherical representation. Roughly speaking, for every bounded strongly linearly convex domain $\Omega\subset\mathbb C^n\, (n>1)$ with $C^3$-smooth boundary and $p\in \partial\Omega$, we can define a special homeomorphism  between its closure $\overline{\Omega}$ and the closed unit ball $\overline{\mathbb B}^n\subset {\mathbb C}^n$, which maps holomorphically each complex geodesic disc of $\Omega$ through $p$
to a complex geodesic disc of ${\mathbb B}^n$ through $\nu_p$, and preserves the corresponding horospheres and non-tangential approach regions; see Sections \ref{subsect:spherical rep} and \ref{sect:CMA equations} for more details. By means of such a boundary spherical representation, we can solve the following homogeneous complex Monge-Amp\`{e}re equation:

\begin{theorem}\label{mainthm2}
Let $\Omega\subset\mathbb C^n\, (n>1)$ be a bounded strongly linearly convex domain with $C^3$-smooth boundary, and let $p\in\partial\Omega$. Then the complex Monge-Amp\`{e}re equation
\begin{equation}\label{MA-bp}
\begin{cases}
u\in {\rm Psh}(\Omega)\cap
L_{{\rm loc}}^{\infty}(\Omega), \\
(dd^c u)^n=0 \quad \quad  \!\quad\ \ {\rm on}\,\ \Omega, \\
u<0\qquad \qquad\quad   \!\quad\ {\rm on}\,\ \Omega, \\
\lim_{z\to x}u(z)=0  \!\quad\quad  {\rm for}\,\ x\in\partial\Omega\setminus\{p\}, \\
u(z)\approx -|z-p|^{-1} \:\:\,\,\, {\rm as} \ \ z\to p\, \ {\rm
nontangentially}
\end{cases}
\end{equation}
admits a solution $P_{\Omega,\, p}\in C(\overline{\Omega}\setminus\{p\})$ whose sub-level sets are
horospheres of $\Omega$ with center $p$. Here the last condition in \eqref{MA-bp} requires that for every $\beta>1$, there exists a constant $C_{\beta}>1$ such that
$$
C_{\beta}^{-1} <-u(z)|z-p|<C_{\beta}
$$
for all $z\in \Gamma_{\beta}(p)$ sufficiently close to $p$, where
\begin{equation}\label{defn:nontangential}
\Gamma_{\beta}(p):=\big\{z\in\Omega\!: |z-p|<\beta\, {\rm
dist}(z,\partial\Omega)\big\}.
\end{equation}
\end{theorem}

Here, ${\rm dist}(\,\cdot\,,\partial\Omega)$ denotes the Euclidean distance to the boundary $\partial\Omega$, and ${\rm Psh}(\Omega)$  the set of plurisubharmonic functions on $\Omega$. The precise definition of horospheres in the sense of  Abate will be given in  Section \ref{sect:CMA equations}. Incidentally, Theorem \ref{mainthm2} has been generalized by Bracci-Saracco-Trapani to bounded strongly pseudoconvex domains in $\mathbb C^n$ (with $C^{\infty}$-smooth boundary); see \cite{Bracci-ST} for details.

In his famous  paper \cite{Lempert81} and later work \cite{Lempert83, Lempert86}, Lempert solved the following homogeneous complex Monge-Amp\`{e}re equation on strongly linearly convex domains $\Omega\subset\mathbb C^n$ with $C^{m,\,\alpha}$-smooth boundary, where $m\ge 2$,  $\alpha \in (0, 1)$ and $w\in \Omega$:
\begin{equation}\label{MA-lem}
\begin{cases}
u\in {\rm Psh}(\Omega)\cap
L_{{\rm loc}}^{\infty}(\Omega\setminus\{w\}), \\
(dd^c u)^n=0 \qquad \qquad  \quad \!\quad\ \ \:\, {\rm on}\,\ \Omega \setminus\{w\}, \\
\lim_{z\to x}u(z)=0  \qquad \qquad \quad\!  \: {\rm for}\,\ x\in\partial\Omega, \\
u(z)-\log|z-w|=O(1)\ \ \;  {\rm as} \ \ z\to w.\\
\end{cases}
\end{equation}
By establishing a singular foliation  with complex geodesic discs passing through  $w\in \Omega$ as its holomorphic leaves, Lempert obtained  a solution to equation \eqref{MA-lem} that is
$C^{m-1,\,\alpha -\varepsilon}$-smooth on $\overline{\Omega}\setminus\{w\}$ for $0<\epsilon<<1$.
Chang-Hu-Lee \cite{Chang-Hu-Lee88} generalized Lempert's work to obtain the holomorphic foliation by complex geodesic discs initiated from  a boundary point $p\in \partial \Omega$, when $\partial\Omega$ is at least $C^{14}$-smooth. By means of this foliation together with many new ideas, Bracci-Patrizio \cite{Bracci-MathAnn} first studied equation \eqref{MA-bp} when $\Omega$ is strongly convex with $C^{m,\,\alpha}$-smooth boundary for $m\ge 14$. They obtained a solution that is $C^{m-4,\,\alpha}$-smooth on
$\overline{\Omega}\setminus\{p\}$, though they only stated the result for $m=\infty$. To study such a foliation based on a boundary point when the boundary of the domain has  minimal regularity, one is led to the construction of complex geodesics introduced in \cite{Huang-Pisa94}. However, to make the construction there workable, one first needs to solve the uniqueness problem of complex geodesics with prescribed boundary data, which is a main content of Theorem \ref{mainthm1}.

We proceed by remarking  that the $C^3$-regularity of $\partial\Omega$ seems to be the optimal regularity in the theory of complex geodesics, cf. \cite{Lempert81, Lempert84, Lempert86, Chang-Hu-Lee88, Chirka83, Chirka-CS99, Huang-Pisa94, KW13}. Also compared with \cite{Chang-Hu-Lee88, Bracci-MathAnn, Bracci-Trans},
our argument in this paper uses the boundary regularity of complex geodesics and their dual mappings in a symmetric way, so that $C^3$-regularity of $\partial\Omega$ is enough; see Sections \ref{sect:uniqueness of geodesics} and  \ref{subsect:spherical rep} for details. Our solution $P_{\Omega,\,p}$ to equation \eqref{MA-bp} is the pullback of the so-called  pluricomplex Poisson kernel on the open unit ball in $\mathbb C^n$ via the aforementioned  boundary spherical representation. It is desirable to get a better relationship
between the regularity of $P_{\Omega,\,p}$ and that of $\partial\Omega$, which will be left to a future investigation to avoid this paper being too long.  Also, it is well worth answering the  following fairly natural and interesting question concerning  Theorem \ref{mainthm2}:

\begin{question} \label{0100}
Is there only one solution (up to a positive constant multiple) to equation \eqref{MA-bp}?
\end{question}

When $\Omega$ is a bounded strongly {\it convex} domain in $\mathbb C^n\, (n>1)$ with  $C^{\infty}$-smooth boundary, Bracci-Patrizio-Trapani \cite{Bracci-Trans} proved that  other solutions to  equation \eqref{MA-bp} must be the positive constant multiples of the one they constructed if they share some common analytic or geometric features. However, it seems  difficult to answer the above question in full generality. In contrast, the uniqueness of solutions to equation \eqref{MA-lem} is relatively easy and follows immediately from the well-known comparison principle for the complex Monge-Amp\`{e}re operator, proved by Bedford-Taylor \cite{Bedford82}. A partial answer to the above question will be observed in Section \ref{sect:CMA equations}. We also refer the interested reader to \cite[Question 7.6]{Bracci-Trans} for a related but more general question posed for bounded strongly convex domains in $\mathbb C^n$ with $C^{\infty}$-smooth boundary.

This paper is organized as follows. In Section \ref{sect:uniqueness of geodesics}, we first prove a quantitative version of the Burns-Krantz rigidity theorem. We then study the boundary regularity of the Lempert left inverse of complex geodesics. Theorem \ref{mainthm1} is eventually proved by using these results together with some technical estimates. Section \ref{subsect:spherical rep} is devoted to the construction and the study of a new boundary spherical representation for bounded strongly linearly convex domains in $\mathbb C^n\, (n>1)$ only with $C^3$-smooth boundary. Finally, Theorem \ref{mainthm2} is  proved in Section \ref{sect:CMA equations}.

\section{ Uniqueness of complex geodesics with prescribed boundary data}\label{sect:uniqueness of geodesics}
This section is devoted to the proof of Theorem \ref{mainthm1}. We begin by presenting the following  version of the well-known Burns-Krantz rigidity theorem. For earlier and very recent related work, see \cite{Chang-Hu-Lee88, Burns-Krantz, Huang-Canad95, BZZ06, Shoikhet, Zimmer18,Bracci20}, etc.

\begin{lemma}\label{thm:effectiveness of BK}
Let $f$ be a holomorphic self-mapping of $\Delta$ such that
\begin{equation}\label{assumption in BK}
f(\zeta_k)=\zeta_k+O(|\zeta_k-1|^3)
\end{equation}
as $k\to \infty$, where $\{\zeta_k\}_{k\in\mathbb N}$ is a sequence in $\Delta$ converging non-tangentially to $1$. Then
\begin{enumerate}[leftmargin=2.0pc, parsep=4pt]
\item  [{\rm (i)}]
    $$
    {\rm Re}\bigg(\frac{f(\zeta)-\zeta}{(\zeta-1)^2}\bigg)\geq 0,\quad \zeta\in\Delta.
    $$
\item  [{\rm (ii)}] $f'''$ admits a non-tangential limit at $1$, denoted by $f'''(1)$, which is a non-positive real number and satisfies the following inequality:
    $$
    |f(\zeta)-\zeta|^2\leq -\frac 13f'''(1)\frac{|1-\zeta|^6}{1-|\zeta|^2}
    {\rm Re}\bigg(\frac{f(\zeta)-\zeta}{(\zeta-1)^2}\bigg),\quad \zeta\in\Delta.
    $$
    In particular, $f$ is the identity if and only if $f'''(1)=0$.
\end{enumerate}
\end{lemma}

Burns-Krantz rigidity theorem was first proved and generalized in \cite{Burns-Krantz}, \cite{Huang-Canad95}, respectively. A version of the Burns-Krantz theorem with the notation $``O"$ in assumption $\eqref{assumption in BK}$ replaced by $``o"$  was first proved by Baracco-Zaitsev-Zampieri in \cite{BZZ06}. With more regularity assumptions about $f$ at $1$, Lemma \ref{assumption in BK} (i) was obtained earlier in \cite{Shoikhet}. A version of Lemma \ref{assumption in BK} (ii) was also discussed in the same paper;
see \cite[Corollary 7]{Shoikhet}.  However, the inequality obtained there is incorrect as the following example shows:

Let $f$ be the holomorphic self-mapping of $\Delta$ given by
\[
f(\zeta)=\frac{10\zeta+(1-\zeta)^2}{10+(1-\zeta)^2}.
\]
Then it is easy to see that $f$ satisfies the assumption in \cite[Corollary 7]{Shoikhet}. Note also that the function ${\rm Re}\left((\zeta-f(\zeta))(1-\overline\zeta)^2\right)$ is negative, rather than positive as stated there. Now  if the estimate in \cite[Corollary 7]{Shoikhet} would hold even after correcting the sign, we would have the following inequality:
    $$
    |f(\zeta)-\zeta|^2\leq -\frac 16f'''(1)\frac{{\rm Re}\big((f(\zeta)-\zeta)(1-\overline{\zeta})^2\big)}{1-|\zeta|^2}
    ,\quad \zeta\in\Delta.
    $$
But evaluating both sides at $\zeta=-1/3$, we see that the preceding
inequality is incorrect.

We now move to the proof of Lemma \ref{assumption in BK}, which is very short and self-contained.

\begin{proof}[Proof of Lemma $\ref{assumption in BK}$]
If $f={\rm Id_{\Delta}}$, then there is nothing to prove. So we next assume that $f\neq {\rm Id_{\Delta}}$. Note that
$$
\lim_{k\to\infty}\frac{1-|f(\zeta_k)|}{1-|\zeta_k|}=1.
$$
By the Julia-Wolff-Cararth\'{e}odory theorem (see, e.g., \cite[Section 1.2.1]{Abatebook}, \cite[Chapter VI]{Sarasonbook}), the quotient $(f(\zeta)-1)/(\zeta-1)$ tends to $1$ as $\zeta\to 1$ non-tangentially. Moreover,
\begin{equation}\label{ineq:JWC01}
\frac{|1-f(\zeta)|^2}{1-|f(\zeta)|^2}\leq\frac{|1-\zeta|^2}{1-|\zeta|^2},
 \quad \zeta\in\Delta.
\end{equation}
Now we consider the holomorphic function $g\!:\Delta\to\mathbb C$ given by
$$
g(\zeta):=\frac{1+f(\zeta)}{1-f(\zeta)}-\frac{1+\zeta}{1-\zeta}.
$$
Then inequality \eqref{ineq:JWC01} implies that $g$ maps $\Delta$ into the closed right half-plane. Since $f\neq {\rm Id_{\Delta}}$, by the maximum principle applied to $-{\rm Re}\,g$ we have that $g(\Delta)$ is
contained in the right half-plane. In particular,
$$
(1-\zeta)g(\zeta)+2\neq 0, \quad \zeta\in\Delta.
$$
On the other hand, since
\begin{equation}\label{ineq:JWC02}
\varphi(\zeta):=\frac{f(\zeta)-\zeta}{(\zeta-1)^2}=\frac{g(\zeta)}{(1-\zeta)g(\zeta)+2},
\end{equation}
we see that
$$
{\rm Re}\,\varphi(\zeta)=\frac{|g(\zeta)|^2{\rm Re}\,(1-\zeta)+2{\rm
Re}\, g(\zeta)}{|(1-\zeta)g(\zeta)+2|^2}>0, \quad \zeta\in\Delta.
$$
This completes the proof of ${\rm (i)}$.

Next, set
\begin{equation}\label{ineq:JWC03}
\psi:=\frac{1-\varphi}{1+\varphi}.
\end{equation}
Then $\psi(\Delta)\subseteq \Delta$, and
\begin{equation}\label{ineq:JWC04}
1-|\psi|^2=\frac{4 {\rm Re}\,\varphi}{|1+\varphi|^2}.
\end{equation}
Together with \eqref{assumption in BK} and \eqref{ineq:JWC02}, this implies that
$$
\liminf_{\Delta\ni\zeta\to
1}\frac{1-|\psi(\zeta)|^2}{1-|\zeta|^2}<+\infty.
$$
By applying the Julia-Wolff-Cararth\'{e}odory theorem again, we see that $(\psi(\zeta)-1)/(\zeta-1)$ admits a non-tangential limit at $1$, denoted by $\psi'(1)$, which is a positive real number and satisfies that
\begin{equation}\label{ineq:JWC05}
\frac{|1-\psi(\zeta)|^2}{1-|\psi(\zeta)|^2}\leq\psi'(1)\frac{|1-\zeta|^2}{1-|\zeta|^2},
 \quad \zeta\in\Delta.
\end{equation}
Furthermore, we can conclude that
$$
\frac{f(\zeta)-\zeta}{(\zeta-1)^3}=\frac{\varphi(\zeta)}{\zeta-1}
=-\frac{1}{1+\psi(\zeta)}\frac{\psi(\zeta)-1}{\zeta-1}\to -\frac12\psi'(1)
$$
as $\zeta\to 1$ non-tangentially. Together with a standard argument using the Cauchy integral formula, this also implies that $f'''$ has a non-tangential limit $f'''(1)$  at $1$, and
\begin{equation}\label{ineq:JWC06}
f'''(1)=-3\psi'(1)<0.
\end{equation}
Now the desired inequality follows immediately by substituting
\eqref{ineq:JWC02}-\eqref{ineq:JWC04} and \eqref{ineq:JWC06} into
\eqref{ineq:JWC05}.
\end{proof}

We next prove Proposition \ref{prop:regularity of leftinv}, which is crucial for our subsequent arguments. This proposition might be known to experts. Since being unable to locate a good reference for its proof, we will give a detailed argument for the reader's convenience. To this end, we need to recall some known results obtained by Lempert \cite{Lempert81, Lempert84}. Let $\Omega\subset\mathbb C^n\, (n>1)$ be a bounded strongly linearly convex domain with $C^{m,\, \alpha}$-smooth boundary, where $m\ge 2$ and $\alpha\in(0, 1)$. Let $\varphi$ be a complex geodesic of $\Omega$, and $\varphi^{\ast}$ be its dual mapping. Then by \cite[Theorem 2]{Lempert84} (see also \cite[Theorem 1.14]{KW13}), the winding number of the function
$$
\partial\Delta\ni \zeta\mapsto \langle z-\varphi(\zeta),\,
 \overline{\varphi^{\ast}(\zeta)}
\rangle
$$
is one for all $z\in\Omega$. Hence for every $z\in\Omega$,  the equation
\begin{equation*}\label{eq:left-inverse of CGs}
\langle z-\varphi(\zeta),\, \overline{\varphi^{\ast}(\zeta)}\rangle=0
\end{equation*}
admits a unique solution $\zeta:=\varrho(z)\in\Delta$. We denote by $\varrho\!: \Omega\to \Delta$  the function defined in such a way, which is uniquely determined by $\varphi$ and  is holomorphic such that
$\varrho\circ\varphi={\rm Id}_{\Delta}$. If we set
\begin{equation*}\label{def:lem-proj}
\rho:=\varphi\circ\varrho,
\end{equation*}
then $\rho\in\mathcal{O}(\Omega,\,\Omega)$ is a holomorphic retraction of $\Omega$ (i.e., $\rho\circ\rho=\rho$) with image $\rho(\Omega)=\varphi(\Delta)$. In the rest of this paper, we will refer
to $\varrho$ and $\rho$ as the {\it Lempert left inverse of $\varphi$}, and the {\it Lempert retraction associated to $\varphi$}, respectively.

\begin{proposition}\label{prop:regularity of leftinv}
Let $\Omega\subset\mathbb C^n\, (n>1)$ be a bounded strongly linearly convex domain with $C^m$-smooth boundary, where $m\geq 3$. Let $\varphi$ be a complex geodesic of $\Omega$ and $\varrho$ be the Lempert left inverse of $\varphi$. Then $\varrho\in \mathcal{O}(\Omega,\,\Delta)\cap C^{m-2,\,\alpha}(\overline{\Omega})$ for all $\alpha\in(0, 1)$, and $\varrho(\overline{\Omega}\setminus\varphi(\partial\Delta))\subset\Delta$.
Moreover, for every multi-index $\nu\in\mathbb N^n$ with $|\nu|=m-1$, $\frac{\partial^{|\nu|}\varrho}{\partial z^{\nu}}$ admits a non-tangential limit at every point
$p\in\varphi(\partial\Delta)$, denoted by $\frac{\partial^{|\nu|}\varrho}{\partial z^{\nu}}(p)$; and for every $0<\alpha<1$ and $\beta>1$, there exists a constant $C_{p,\,\alpha,\,\beta}>0$ such that
\begin{equation}\label{ineq:asy-oeder}
   \left|\frac{\partial^{|\nu|}\varrho}{\partial z^{\nu}}(z)-\frac{\partial^{|\nu|}\varrho}{\partial z^{\nu}}(p)\right|\leq C_{p,\,\alpha,\,\beta}|z-p|^{\alpha},
\end{equation}
for all $z\in \Gamma_{\beta}(p)$, which is as in \eqref{defn:nontangential}.
\end{proposition}

\begin{proof}
Let $\varphi^{\ast}$ be the dual mapping of $\varphi$ as before. Then $\varphi$, $\varphi^{\ast}\in C^{m-2,\,\alpha}(\overline{\Delta})$ for all $\alpha\in (0, 1)$. For every $z\in\Omega$, $\zeta:=\varrho(z)\in\Delta$ is by definition the only solution to the equation
\begin{equation}\label{eq:left-inv-eqn}
   \langle z-\varphi(\zeta),\, \overline{\varphi^{\ast}(\zeta)}\rangle=0.
\end{equation}
More explicitly,
\begin{equation}\label{eq:left-inv-for}
   \varrho(z)=\frac1{2\pi i}\int_{\partial\Delta}\zeta\frac{\langle z-\varphi(\zeta),\,
   \overline{(\varphi^{\ast})'(\zeta)}\rangle-1}{\langle z-\varphi(\zeta),\, \overline{\varphi^{\ast}(\zeta)}\rangle}d\zeta, \quad z\in\Omega.
\end{equation}
Obviously, $\varrho\in\mathcal{O}(\Omega,\,\Delta)$. Note that $\varrho$ is initially defined on $\Omega$. Now we show that $\varrho$ can extend $C^{m-2}$-smoothly to $\overline{\Omega}$. Indeed, we first see easily that there is an open set
$U\supset\overline{\Omega}\setminus\varphi(\partial\Delta)$ such that for every $z\in U$, the winding number of the function
$$
\partial\Delta\ni\zeta\mapsto\langle z-\varphi(\zeta),\, \overline{\varphi^{\ast}(\zeta)}
\rangle
$$
is one. Therefore, the right-hand side of \eqref{eq:left-inv-for} determines a  $C^{m-2}$-smooth function from $U$ to $\Delta$, which assigns to every $z\in U$ the only solution to equation \eqref{eq:left-inv-eqn}. This means that $\varrho$ can extend $C^{m-2}$-smoothly to  $\overline{\Omega}\setminus\varphi(\partial\Delta)$. To check the $C^{m-2}$-smooth extendibility of $\varrho$ to $\varphi(\partial\Delta)$, we first extend $\varphi$ and $\varphi^{\ast}$ $C^{m-2}$-smoothly to $\mathbb C$. We still denote by $\varphi$ and $\varphi^{\ast}$ their extensions, and consider the function
$$
F(z, \zeta):=\langle z-\varphi(\zeta),\,
\overline{\varphi^{\ast}(\zeta)}\rangle, \quad (z, \zeta)\in\mathbb
C^n\times\mathbb C.
$$
Then for every $\zeta_0\in\partial\Delta$, it holds that $F(\varphi(\zeta_0), \zeta_0)=0$ and $\frac{\partial
F}{\partial\zeta}(\varphi(\zeta_0), \zeta_0)=-1$. Thus by the implicit function theorem, there exists a neighborhood $U_{\zeta_0}\times V_{\zeta_0}$ of $(\varphi(\zeta_0), \zeta_0)$ and
a function $\varrho_0\in C^{m-2}(U_{\zeta_0}, V_{\zeta_0})$ such that
$$
\big\{(z, \zeta)\in U_{\zeta_0}\times V_{\zeta_0}\!: F(z, \zeta)=0\big\}=\big\{(z, \varrho_0(z))\!: z\in U_{\zeta_0}\big\}.
$$
Now by uniqueness of the solution to equation \eqref{eq:left-inv-eqn}, we see that $\varrho=\varrho_0$ on
$U_{\zeta_0}\cap(\overline{\Omega}\setminus\varphi(\partial\Delta))$ for all $\zeta_0\in\partial\Delta$. In other words, $\varrho$ can also extend $C^{m-2}$-smoothly to  $\varphi(\partial\Delta)$ as desired.

Now we assume that $m=3$. Differentiating the equality
$$
\langle z-\varphi\circ\varrho(z),\, \overline{\varphi^{\ast}\circ\varrho(z)}\rangle=0
$$
on $\overline{\Omega}$ with respect to $z_j$, and taking into account that $\langle\varphi',\, \overline{\varphi^{\ast}}\rangle=1$, we see that
$$
\frac{\partial \varrho}{\partial z_j}(z)
\Big(1-\langle z-\varphi\circ\varrho(z),\, \overline{(\varphi^{\ast})'\circ\varrho(z)}\rangle\Big)
=\langle e_j,\, \overline{\varphi^{\ast}\circ\varrho(z)}\rangle
$$
for all $z\in\overline{\Omega}$. Since $\varphi^{\ast}$ is nowhere vanishing on $\overline{\Delta}$, the preceding equality implies that
\begin{equation}\label{ineq:nonvanishing}
\inf_{z\in\overline{\Omega}}
\big|1-\langle z-\varphi\circ\varrho(z),\,
\overline{(\varphi^{\ast})'\circ\varrho(z)}\rangle\big|>0,
\end{equation}
and thus
\begin{equation}\label{eq:formular of left-inverse}
\frac{\partial \varrho}{\partial z_j}(z)=
\frac{\langle e_j,\, \overline{\varphi^{\ast}\circ\varrho(z)}\rangle}
{1-\langle z-\varphi\circ\varrho(z),\, \overline{(\varphi^{\ast})'\circ\varrho(z)}\rangle},
\quad z\in\overline{\Omega}.
\end{equation}
In particular, this implies that
\begin{equation}\label{eq: re-dual-lefinv}
\varphi^{\ast}=\left(\frac{\partial \varrho}{\partial z_1}\circ\varphi,\ldots,\frac{\partial \varrho}{\partial z_n}\circ\varphi\right).
\end{equation}
Moreover, from \eqref{ineq:nonvanishing}, \eqref{eq:formular of left-inverse} and the regularity of $\varphi$, $\varphi^{\ast}$ it follows that $\varrho\in  C^{1,\,\alpha}(\overline{\Omega})$ for all $\alpha\in(0, 1)$.
Differentiating equality \eqref{eq:formular of left-inverse} once again yields that
\begin{equation}\label{eq:second-der-leftinv}
\begin{split}
\frac{\partial^2 \varrho}{\partial z_j\partial z_k}&
(z)=\frac{\frac{\partial \varrho}{\partial z_k}(z)
\langle e_j,\, \overline{(\varphi^{\ast})'\circ\varrho(z)}\rangle}
{1-\langle z-\varphi\circ\varrho(z),\, \overline{(\varphi^{\ast})'\circ\varrho(z)}\rangle}
+\frac{\langle e_j,\, \overline{\varphi^{\ast}\circ\varrho(z)}\rangle}
{\big(1-\langle z-\varphi\circ\varrho(z),\, \overline{(\varphi^{\ast})'\circ\varrho(z)}\rangle\big)^2}\cdot
\\
&\Big(
\big\langle e_k-\frac{\partial \varrho}{\partial z_k}(z)\varphi'\circ\varrho(z),\, \overline{(\varphi^{\ast})'\circ\varrho(z)}\big\rangle+\frac{\partial \varrho}{\partial z_k}(z)\big\langle z-\varphi\circ\varrho(z),\, \overline{(\varphi^{\ast})''\circ\varrho(z)}\big\rangle
\Big)
\end{split}
\end{equation}
for all $z\in\Omega$, and $1\leq j, \, k\leq n$. Let $p\in\varphi(\partial\Delta)$, and take $0<\alpha<1$, $\beta>1$. We now need to show that $\frac{\partial^2 \varrho}{\partial z_j\partial z_k}$ admits a non-tangential limit at $p$ and satisfies estimate $\eqref{ineq:asy-oeder}$. In view of \eqref{ineq:nonvanishing}, \eqref{eq:second-der-leftinv}, and the fact that $\varphi$, $\varphi^{\ast}\in C^{1,\,\alpha}(\overline{\Delta})$ and $\varrho\in  C^{1,\,\alpha}(\overline{\Omega})$, it suffices to prove that there exists a constant $C_{p,\,\alpha,\,\beta}>0$ such that
$$
\big|\langle z-\varphi\circ\varrho(z),\, \overline{(\varphi^{\ast})''\circ\varrho(z)}\rangle\big|
\leq C_{p,\,\alpha,\,\beta}|z-p|^{\alpha}
$$
for all $z\in\Gamma_{\beta}(p)$. To this end, it first follows from the Hopf lemma (see, e.g., \cite[Proposition 12.2]{FornaessSCV}) that
$$
1-|\varrho|\geq C{\rm dist}(\,\cdot\,,\partial\Omega)$$
for some constant $C>0$. Now the desired result follows immediately by applying the classical Hardy-Littlewood theorem to $(\varphi^{\ast})'\in \mathcal{O}(\Delta)\cap C^{\alpha}(\overline{\Delta})$. This completes the proof of the case when $m=3$, and the general case follows in an analogous way.
\end{proof}

\begin{remark}\label{rmk:uniqueness of solution}
From the above proof it follows immediately that for every $z\in\overline{\Omega}$, the equation $\langle z-\varphi(\zeta),\, \overline{\varphi^{\ast}(\zeta)}\rangle=0$ admits a unique solution
on $\overline{\Delta}$, which is precisely $\varrho(z)$.
\end{remark}

Now, we are ready to prove Theorem $\ref{mainthm1}$.

\begin{proof}[Proof of Theorem $\ref{mainthm1}$]
As we mentioned in the Introduction, the existence part was already known. So we need only prove the uniqueness part.

Suppose that $\varphi$ and $\widetilde{\varphi}$ are two complex geodesics of $\Omega$ such that $\varphi(1)=\widetilde{\varphi}(1)=p$ and $\varphi'(1)=\widetilde{\varphi}'(1)=v$. For every $t\in\mathbb R$, set
\begin{equation}\label{t-parameter}
\sigma_t(\zeta):=\frac{(1-it)\zeta+it}{-it\zeta+1+it}\in {\rm{Aut}}(\Delta).
\end{equation}
Then we need to prove that there exists a $t_0\in\mathbb R$ such that
$\widetilde{\varphi}=\varphi\circ\sigma_{t_0}$.

We denote by $\varphi^{\ast}$, $\widetilde{\varphi}^{\ast}$ the dual mappings of $\varphi$ and $\widetilde{\varphi}$, respectively. Then $\varphi$, $\widetilde{\varphi}$, $\varphi^{\ast}$,
$\widetilde{\varphi}^{\ast}\in C^{1,\,\alpha}(\overline{\Delta})$ for all $\alpha\in (0, 1)$. Next, we show that there exists a $t_0\in\mathbb R$ such that
\begin{equation}\label{eq:key:normal parameter}
((\varphi\circ\sigma_{t_0})^{\ast})'(1)=(\widetilde{\varphi}^{\ast})'(1).
\end{equation}
To this end, we first write
\begin{equation}\label{eq:expression of dual 1}
\varphi^{\ast}|_{\partial\Delta}(\zeta)=\zeta\mu(\zeta)\overline{\nu\circ\varphi(\zeta)},
\quad\quad
\widetilde{\varphi}^{\ast}|_{\partial\Delta}(\zeta)
=\zeta\widetilde{\mu}(\zeta)\overline{\nu\circ\widetilde{\varphi}(\zeta)},
\end{equation}
and
\begin{equation}\label{eq:expression of dual 2}
(\varphi\circ\sigma_t)^{\ast}|_{\partial\Delta}(\zeta)
=\zeta\mu_t(\zeta)\overline{\nu\circ\varphi\circ\sigma_t(\zeta)},
\end{equation}
where $\mu$, $\mu_t$, $\widetilde{\mu}$ are $C^{1,\,\alpha}$-smooth positive functions on $\partial\Delta$, and $\nu$ denotes the unit  outward normal vector field of $\partial\Omega$. Now note that
$$
(\varphi\circ\sigma_t)'(1)=\varphi'(1)=\widetilde{\varphi}'(1),
$$
and hence
$$
\left.\frac{d}{d\theta}\right|_{\theta=0}\nu\circ\varphi\circ\sigma_t(e^{i\theta})
=\left.\frac{d}{d\theta}\right|_{\theta=0}\nu\circ\widetilde{\varphi}(e^{i\theta})
$$
for all $t\in\mathbb R$. Thus in view of \eqref{eq:expression of dual 1} and \eqref{eq:expression of dual 2}, \eqref{eq:key:normal parameter} is equivalent to
\begin{equation}\label{eq:key:equi normal parameter}
\left.\frac{d}{d\theta}\right|_{\theta=0}\mu_{t_0}(e^{i\theta})
=\left.\frac{d}{d\theta}\right|_{\theta=0}\widetilde{\mu}(e^{i\theta}).
\end{equation}
Now by the definition of dual mappings,
\begin{equation*}
\begin{split}
\mu_t(e^{i\theta}) &= e^{-i\theta}\big\langle(\varphi\circ\sigma_t)'(e^{i\theta}),\,
\nu\circ\varphi\circ\sigma_t(e^{i\theta})\big\rangle^{-1}\\
&=\frac{e^{-i\theta}}{\sigma'_t(e^{i\theta})}
\big\langle\varphi'\circ\sigma_t(e^{i\theta}),\,                   \nu\circ\varphi\circ\sigma_t(e^{i\theta})\big\rangle^{-1}\\
&=\frac{\sigma_t(e^{i\theta})}{e^{i\theta}\sigma'_t(e^{i\theta})}\mu\circ\sigma_t(e^{i\theta})
\end{split}
\end{equation*}
for all $\theta\in\mathbb R$. Moreover, since $\sigma_t(1)=\sigma'_t(1)=1$, it follows that
$$
\left.\frac{d}{d\theta}\right|_{\theta=0}\mu\circ\sigma_t(e^{i\theta})
=\left.\frac{d}{d\theta}\right|_{\theta=0}\mu(e^{i\theta}).
$$
We then conclude by a direct calculation that
\begin{equation*}\label{t-parameter22}
\begin{split}
\left.\frac{d}{d\theta}\right|_{\theta=0}\mu_t(e^{i\theta})
=2t\mu(1)+\left.\frac{d}{d\theta}\right|_{\theta=0}\mu(e^{i\theta})
=\frac{2t}{\langle v,
\nu_p\rangle}+\left.\frac{d}{d\theta}\right|_{\theta=0}\mu(e^{i\theta}),
\end{split}
\end{equation*}
which implies that \eqref{eq:key:equi normal parameter} (and hence \eqref{eq:key:normal parameter}) holds provided
$$
t_0=\frac12\langle v, \nu_p\rangle
\left.\frac{d}{d\theta}\right|_{\theta=0}\big(\widetilde{\mu}(e^{i\theta})-\mu(e^{i\theta})\big).
$$

\medskip
Now we come to show that $\widetilde{\varphi}=\varphi\circ\sigma_{t_0}$.  We first give a proof for the strongly convex case as a warmup. We argue by contradiction, and suppose on the contrary that $\widetilde{\varphi}\neq\varphi\circ\sigma_{t_0}$. Then we have
$$
\widetilde{\varphi}(\overline{\Delta})\cap\varphi\circ\sigma_{t_0}(\overline{\Delta})=\{p\},
$$
since two different closed complex geodesic discs can have at most
one point in common; see \cite[pp. 362--363]{Lempert84}. Together with the strong convexity of
$\Omega$, this further implies that
$$
{\rm Re}\big\langle
\widetilde{\varphi}(\zeta)-\varphi\circ\sigma_{t_0}(\zeta),
\, \overline{\zeta^{-1}\widetilde{\varphi}^{\ast}(\zeta)}
\big\rangle=
\widetilde{\mu}(\zeta){\rm Re}
\big\langle
\widetilde{\varphi}(\zeta)-\varphi\circ\sigma_{t_0}(\zeta), \,\nu\circ\widetilde{\varphi}(\zeta)
\big\rangle>0
$$
and
$$
{\rm Re}\big\langle
\varphi\circ\sigma_{t_0}(\zeta)-\widetilde{\varphi}(\zeta),
\, \overline{\zeta^{-1}(\varphi\circ\sigma_{t_0})^{\ast}(\zeta)}
\big\rangle>0
$$
hold on $\partial\Delta\setminus\{1\}$. Taking summation yields that
\begin{equation}\label{geodesic:key ineq1}
{\rm Re}\big\langle
\widetilde{\varphi}(\zeta)-\varphi\circ\sigma_{t_0}(\zeta), \,
\overline{\zeta^{-1}\big(\widetilde{\varphi}^{\ast}(\zeta)-      (\varphi\circ\sigma_{t_0})^{\ast}(\zeta)\big)}
\big\rangle>0
\end{equation}
for all $\zeta\in\partial\Delta\setminus\{1\}$. Note also that for every $\zeta\in\partial\Delta\setminus\{1\}$,
$$
\frac{\zeta}{(1-\zeta)^2}=\frac14\left(\left(\frac{1+\zeta}{1-\zeta}\right)^2-1\right)\leq-\frac14,
$$
we therefore deduce from \eqref{geodesic:key ineq1} that
\begin{equation}\label{geodesic:key ineq2}
{\rm Re}
\bigg\langle
\frac{\zeta\big(\widetilde{\varphi}(\zeta)-\varphi\circ\sigma_{t_0}(\zeta)\big)}{(1-\zeta)^2},\,
\overline{\frac{\widetilde{\varphi}^{\ast}(\zeta)-(\varphi\circ\sigma_{t_0})^{\ast}(\zeta)}{(1-\zeta)^2}}
\bigg\rangle>0
\end{equation}
on $\partial\Delta\setminus\{1\}$.
On the other hand, since $\widetilde{\varphi}$, $\varphi\circ\sigma_{t_0}$, $\widetilde{\varphi}^{\ast}$,  $(\varphi\circ\sigma_{t_0})^{\ast}\in C^{1,\,\alpha}(\overline{\Delta})$ for $\alpha>1/2$, and
$$
(\varphi\circ\sigma_{t_0})'(1)=\widetilde{\varphi}'(1),
\quad\quad
((\varphi\circ\sigma_{t_0})^{\ast})'(1)=(\widetilde{\varphi}^{\ast})'(1),
$$
we see that the holomorphic function
$$
f(\zeta):=\zeta\bigg\langle
\frac{\widetilde{\varphi}(\zeta)-\varphi\circ\sigma_{t_0}(\zeta)}{(1-\zeta)^2},\,
\overline{\frac{\widetilde{\varphi}^{\ast}(\zeta)-(\varphi\circ\sigma_{t_0})^{\ast}(\zeta)}{(1-\zeta)^2}}
\bigg\rangle
$$
belongs to the Hardy space $H^1(\Delta)$. Together with \eqref{geodesic:key ineq2}, this implies that
$$
0={\rm Re}f(0)={\rm Re}\bigg(\frac1{2\pi i}\int_{\partial\Delta}\frac{f(\zeta)}{\zeta}d\zeta\bigg)
=\frac1{2\pi}\int_{0}^{2\pi}{\rm Re}f(e^{i\theta})d\theta>0.
$$
This is a contradiction. Therefore, we must have $\widetilde{\varphi}=\varphi\circ\sigma_{t_0}$.

\medskip
We now turn to the strongly linearly convex case. The proof of this general case is much more involved  than that of the strongly convex case. We argue as follows. Let $\varrho$ be the Lempert left inverse of $\varphi$. Then in view of Proposition \ref{prop:regularity of leftinv}, $\varrho\in \mathcal{O}(\Omega,\, \Delta)\cap C^{1,\,\alpha}(\overline{\Omega})$ for all $\alpha\in(0, 1)$. We now consider the holomorphic function $\varrho\circ\widetilde{\varphi}$, which is in $C^{1,\,\alpha}(\overline{\Delta})$ for all $\alpha\in(0, 1)$, and satisfies that
$$
\varrho\circ\widetilde{\varphi}(1)=\varrho\circ\varphi(1)=1.
$$
By differentiating, we obtain
$$
(\varrho\circ\widetilde{\varphi})'(\zeta)=\sum_{j=1}^{n}\frac{\partial \varrho}{\partial z_j}\circ\widetilde{\varphi}(\zeta)\widetilde{\varphi}_j'(\zeta)
=\langle \widetilde{\varphi}'(\zeta),\, \overline{({\rm gard
}\,\varrho)\circ\widetilde{\varphi}(\zeta)}\rangle
$$
for all $\zeta\in\overline{\Delta}$, where
$$
({\rm gard }\,\varrho)(z)=\frac{\partial \varrho}{\partial z}(z):
=\left(\frac{\partial \varrho}{\partial z_1}(z),\ldots,\frac{\partial \varrho}{\partial z_n}(z)\right).
$$
In what follows, to simplify the notation, we assume  that the number $t_0$ in \eqref{eq:key:normal parameter} is zero. Then
\begin{equation}\label{eq:equi-der}
(\varphi^{\ast})^{\prime}(1)=(\widetilde{\varphi}^{\ast})^{\prime}(1).
\end{equation}
Moreover, since $\widetilde{\varphi}(1)=\varphi(1)=p$,
 it follows from \eqref{eq: re-dual-lefinv} and the definition of dual mappings  that
\begin{equation}\label{eq:contact01}
({\rm gard }\,\varrho)\circ\widetilde{\varphi}(1)=\widetilde{\varphi}^{\ast}(1)=\varphi^{\ast}(1)
\end{equation}
and
$$
(\varrho\circ\widetilde{\varphi})'(1)
=\langle \widetilde{\varphi}'(1),\, \overline{({\rm gard }\,\varrho)\circ\widetilde{\varphi}(1)}\rangle
=\langle \widetilde{\varphi}'(1),\, \overline{\widetilde{\varphi}^{\ast}(1) }\rangle
=1.
$$

Let $\widetilde{\varrho}$ be the Lempert left inverse of $\widetilde{\varphi}$. We next claim that
\begin{equation}\label{eq:key point of uniqueness}
(\varrho\circ\widetilde{\varphi}+\widetilde{\varrho}\circ\varphi)''(\zeta)=o(|\zeta-1|)
\end{equation}
as $\zeta\to 1$ non-tangentially. This is  the  main part of the proof. First of all, we have
\begin{equation}\label{eq:second der}
(\varrho\circ\widetilde{\varphi})''(\zeta)
=\sum_{j,\, k=1}^{n}\frac{\partial^2 \varrho}{\partial z_j\partial z_k}\circ\widetilde{\varphi}(\zeta)\widetilde{\varphi}_j'(\zeta)\widetilde{\varphi}_k'(\zeta)
+\langle \widetilde{\varphi}''(\zeta),\, \overline{({\rm gard }\,\varrho)\circ\widetilde{\varphi}(\zeta)}\rangle
\end{equation}
for all $\zeta\in\Delta$. Now we try to estimate the second term. To this end,  differentiating both sides of the identity $\varphi^{\ast}=({\rm gard }\,\varrho)\circ\varphi$ (see \eqref{eq: re-dual-lefinv}) yields that
\begin{equation}\label{der-dual}
(\varphi^{\ast})'(\zeta)
=\sum_{k=1}^{n}\frac{\partial^2 \varrho}{\partial z
\partial z_k}\circ\varphi(\zeta)\varphi_k'(\zeta), \quad \zeta\in\Delta.
\end{equation}
For every $\beta>1$, we set
\begin{equation*}
\mathcal{R}_{\beta}:=\big\{\zeta\in \Delta\!: |\zeta-1|<\beta(1-|\zeta|)\big\},
\end{equation*}
which is a non-tangential approach region in $\Delta$ with vertex
$1$ and aperture $\beta$, called a {\it Stolz region}. We show that
for every $\alpha\in(0, 1)$ and every $\beta>1$,
\begin{equation}\label{ineq:estimate00}
\big|(({\rm gard }\,\varrho)\circ\widetilde{\varphi})'(\zeta)-(\widetilde{\varphi}^{\ast})'(\zeta)\big|
\leq C_{\alpha, \,\beta}|\zeta-1|^{\alpha}
\end{equation}
as $\zeta\in\mathcal{R}_{\beta}$. Here and in what follows, $C_{\alpha}$ (resp. $C_{\alpha,\,\beta}$) always denotes a positive constant depending only on $\alpha$ (resp. $\alpha$ and $\beta$), which could be different in different contexts. Note that $\Omega$ is strongly pseudoconvex, we can take a $C^3$-defining function $r$ for $\Omega$ which is strictly plurisubharmonic on some neighborhood of $\overline{\Omega}$. Then the classical Hopf lemma applied to $r\circ\varphi$ yields that
$$
\inf_{\zeta\in\Delta}\frac{-r\circ\varphi(\zeta)}{1-|\zeta|}>0.
$$
Also, it is evident that
$$
\sup_{z\in\Omega}\frac{-r(z)}{{\rm dist}(z, \partial\Omega)}<\infty.
$$
We then see that there exists a constant $C>0$ such that
$$
{\rm dist}(\varphi(\zeta), \partial\Omega)\geq C(1-|\zeta|)
$$
for all $\zeta\in\Delta$. Consequently, we conclude that $\varphi$ maps every non-tangential approach region in $\Delta$ with vertex $1$ to a non-tangential approach region in $\Omega$ with vertex $p$, and the same
holds true for $\widetilde{\varphi}$. Note also that $\widetilde{\varphi}'(1)=\varphi'(1)$,  we then deduce from  \eqref{ineq:asy-oeder} and  \eqref{der-dual}  (as well as the fact that $\varphi'$, $\widetilde{\varphi}'\in C^{\alpha}(\overline{\Delta})$) that
\begin{equation*}
\begin{split}
\big|(({\rm gard
}\,\varrho)\circ\widetilde{\varphi})'(\zeta)-({\varphi}^{\ast})'(\zeta)\big|
\leq&\, C_{\alpha,
\,\beta}\Big(|\varphi(\zeta)-p|^{\alpha}+|\widetilde{\varphi}(\zeta)-p|^{\alpha}
+|\varphi'(\zeta)-\widetilde{\varphi}'(\zeta)|\Big)\\
\leq&\, C_{\alpha, \,\beta}|\zeta-1|^{\alpha}
\end{split}
\end{equation*}
for all $\zeta\in\mathcal{R}_{\beta}$. Now \eqref{ineq:estimate00} follows immediately,  since $(\varphi^{\ast})^{\prime}(1)=(\widetilde{\varphi}^{\ast})^{\prime}(1)$ (see \eqref{eq:equi-der}) and $(\varphi^{\ast})'$, $(\widetilde{\varphi}^{\ast})'\in C^{\alpha}(\overline{\Delta})$.
Then in view of \eqref{eq:contact01} and \eqref{ineq:estimate00}, we see that
\begin{equation}\label{ineq:estimate01}
\begin{split}
\big|({\rm gard }\,\varrho)\circ\widetilde{\varphi}(\zeta)-\widetilde{\varphi}^{\ast}(\zeta)\big|
\leq&\,|\zeta-1|\int_{0}^{1}\big|(({\rm gard }\,\varrho)\circ\widetilde{\varphi}
-\widetilde{\varphi}^{\ast})'(t\zeta+(1-t))\big|dt\\
\leq&\, C_{\alpha, \,\beta}|\zeta-1|^{\alpha+1}
\end{split}
\end{equation}
for all $\zeta\in\mathcal{R}_{\beta}$. Similarly, we also have
\begin{equation}\label{ineq:estimate02}
\max\Big\{|\varrho\circ\widetilde{\varphi}(\zeta)-\zeta|,\,
\big|\widetilde{\varphi}(\zeta)-\varphi\circ\varrho\circ\widetilde{\varphi}(\zeta)\big|, \,
\big|\widetilde{\varphi}^{\ast}(\zeta)-\varphi^{\ast}\circ\varrho\circ\widetilde{\varphi}(\zeta)\big|
\Big\}\leq C_{\alpha}|\zeta-1|^{\alpha+1}
\end{equation}
for all $\zeta\in\Delta$. Now recall that $\widetilde{\varphi}'\in \mathcal{O}(\Delta)\cap C^{\alpha}(\overline{\Delta})$, the classical Hardy-Littlewood theorem implies that
$$
\sup_{\zeta\in\Delta}(1-|\zeta|)^{1-\alpha}|\widetilde{\varphi}''(\zeta)|<\infty.
$$
Combining this with \eqref{ineq:estimate01}, we then deduce that for
every $\alpha\in(0, 1)$ and every $\beta>1$,
\begin{equation}\label{ineq:estimate03}
\big|\langle \widetilde{\varphi}''(\zeta),\, \overline{({\rm gard }\,\varrho)\circ\widetilde{\varphi}(\zeta)}\rangle
-\langle \widetilde{\varphi}''(\zeta),\, \overline{\widetilde{\varphi}^{\ast}(\zeta)}\rangle\big|
\leq C_{\alpha, \,\beta}|\zeta-1|^{2\alpha}
\end{equation}
as $\zeta\in\mathcal{R}_{\beta}$. Now we deal with the first term in equality \eqref{eq:second der}. A straightforward calculation using \eqref{eq:second-der-leftinv} shows that for every $\zeta\in\Delta$,
\begin{equation*}
\begin{split}
\sum_{j,\, k=1}^{n}
&\frac{\partial^2 \varrho}{\partial z_j\partial z_k}\circ\widetilde{\varphi}(\zeta)\widetilde{\varphi}_j'(\zeta)\widetilde{\varphi}_k'(\zeta)
=\frac{\langle\widetilde{\varphi}'(\zeta),\, \overline{(\varphi^{\ast})'\circ\varrho\circ\widetilde{\varphi}(\zeta)}\rangle
\langle \widetilde{\varphi}'(\zeta),\, \overline{({\rm gard }\,\varrho)\circ\widetilde{\varphi}(\zeta)}\rangle}
{1-\langle \widetilde{\varphi}(\zeta)-\varphi\circ\varrho\circ\widetilde{\varphi}(\zeta),\, \overline{(\varphi^{\ast})'\circ\varrho\circ\widetilde{\varphi}(\zeta)}\rangle}
\\
+&
\frac{\big\langle\widetilde{\varphi}'(\zeta),\, \overline{\varphi^{\ast}\circ\varrho\circ\widetilde{\varphi}(\zeta)}\big\rangle}
{\big(1-\langle \widetilde{\varphi}(\zeta)-\varphi\circ\varrho\circ\widetilde{\varphi}(\zeta),\, \overline{(\varphi^{\ast})'\circ\varrho\circ\widetilde{\varphi}(\zeta)}\rangle\big)^2}
\bigg(
\langle\widetilde{\varphi}'(\zeta), \overline{({\rm gard }\,\varrho)\circ\widetilde{\varphi}(\zeta)}\rangle
\big\langle\widetilde{\varphi}(\zeta)-
\\
&\varphi\circ\varrho\circ\widetilde{\varphi}(\zeta),\,   \overline{(\varphi^{\ast})''\circ\varrho\circ\widetilde{\varphi}(\zeta)}\big\rangle
+
\Big\langle
\widetilde{\varphi}'(\zeta)-\langle \widetilde{\varphi}'(\zeta),\, \overline{({\rm gard }\,\varrho)\circ\widetilde{\varphi}(\zeta)}\rangle
\varphi'\circ\varrho\circ\widetilde{\varphi}(\zeta),\,
\\
&\overline{(\varphi^{\ast})'\circ\varrho\circ \widetilde{\varphi}(\zeta)}
\Big\rangle
\bigg)=:I(\zeta)+II(\zeta).
\end{split}
\end{equation*}
Now in view of \eqref{ineq:nonvanishing}, \eqref{ineq:estimate01}
and \eqref{ineq:estimate02}, and also noticing the arbitrariness of
$\alpha\in (0, 1)$,  it holds that
\begin{eqnarray}\label{ineq:estimate04}
\begin{split}
\big|I(\zeta)-\langle \widetilde{\varphi}'(\zeta),\, \overline{(\varphi^{\ast})'(\zeta)}\rangle\big|
\leq&\, C
\Big(
\big|\widetilde{\varphi}(\zeta)-\varphi\circ\varrho\circ\widetilde{\varphi}(\zeta)\big|
+\big|({\rm gard }\,\varrho)\circ\widetilde{\varphi}(\zeta)-\widetilde{\varphi}^{\ast}(\zeta)\big|
\\
&+\big|(\varphi^{\ast})'\circ\varrho\circ\widetilde{\varphi}(\zeta)-(\varphi^{\ast})'(\zeta)\big|
\Big)
\\
\leq&\, C_{\alpha, \,\beta}
\Big(|\zeta-1|^{\alpha+1}+|\varrho\circ\widetilde{\varphi}(\zeta)-\zeta|^{\frac{2\alpha}{\alpha+1}}\Big)
\\
\leq&\, C_{\alpha, \,\beta} |\zeta-1|^{2\alpha}
\end{split}
\end{eqnarray}
for all $\zeta\in\mathcal{R}_{\beta}$. We next estimate the function $II$. Indeed, a simple manipulation using \eqref{ineq:nonvanishing}, \eqref{ineq:estimate01} and \eqref{ineq:estimate02} again yields that
\begin{equation}\label{ineq:estimate05}
\begin{split}
\big|II(&\zeta)-\langle \widetilde{\varphi}'(\zeta)
-\varphi'\circ\varrho\circ\widetilde{\varphi}(\zeta), \, \overline{(\varphi^{\ast})'\circ\varrho\circ\widetilde{\varphi}(\zeta)}\rangle\big|\\
\leq&\,C
\Big(
\big|\widetilde{\varphi}(\zeta)-\varphi\circ\varrho\circ\widetilde{\varphi}(\zeta)\big|
|(\varphi^{\ast})''\circ\varrho\circ\widetilde{\varphi}(\zeta)|+\big|({\rm gard }\,\varrho)\circ\widetilde{\varphi}(\zeta)-\widetilde{\varphi}^{\ast}(\zeta)\big|
\Big)\\
&+C\big|
\langle\widetilde{\varphi}'(\zeta), \, \overline{\varphi^{\ast}\circ\varrho\circ\widetilde{\varphi}(\zeta)}\rangle
-\big(1-\langle \widetilde{\varphi}(\zeta)-\varphi\circ\varrho\circ\widetilde{\varphi}(\zeta),\, \overline{(\varphi^{\ast})'\circ\varrho\circ\widetilde{\varphi}(\zeta)}\rangle\big)^2
\big|\\
\leq&\,C
\Big(
\big|\widetilde{\varphi}(\zeta)-\varphi\circ\varrho\circ\widetilde{\varphi}(\zeta)\big|
|(\varphi^{\ast})''\circ\varrho\circ\widetilde{\varphi}(\zeta)|+\big|({\rm gard }\,\varrho)\circ\widetilde{\varphi}(\zeta)-\widetilde{\varphi}^{\ast}(\zeta)\big|
\Big)\\
&+C
\Big(
\big|\varphi^{\ast}\circ\varrho\circ\widetilde{\varphi}(\zeta)-\widetilde{\varphi}^{\ast}(\zeta)\big|
+\big|\widetilde{\varphi}(\zeta)-\varphi\circ\varrho\circ\widetilde{\varphi}(\zeta)\big|
\Big)\\
\leq&\, C_{\alpha, \,\beta}
\big(
|\zeta-1|^{\alpha+1}(1-|\varrho\circ\widetilde{\varphi}(\zeta)|)^{\alpha-1}+|\zeta-1|^{\alpha+1}
\big)\\
\leq&\, C_{\alpha, \,\beta} |\zeta-1|^{2\alpha}
\end{split}
\end{equation}
for all $\zeta\in\mathcal{R}_{\beta}$. Here the penultimate inequality follows by applying the classical Hardy-Littlewood theorem to $(\varphi^{\ast})'\in \mathcal{O}(\Delta)\cap C^{\alpha}(\overline{\Delta})$, and the last one follows by noting that
$$
\lim_{\mathcal{R}_{\beta}\ni\zeta\to 1}
\frac{1-|\varrho\circ\widetilde{\varphi}(\zeta)|}{1-|\zeta|}=(\varrho\circ\widetilde{\varphi})'(1)=1,
$$
in view of the Julia-Wolff-Cararth\'{e}odory theorem. For every $\alpha\in(0, 1)$, we also have
\begin{equation*}
\begin{split}
\big|\langle \widetilde{\varphi}'&(\zeta)-\varphi'(\zeta),\, \overline{(\varphi^{\ast})'(\zeta)}\rangle
-\langle \widetilde{\varphi}'(\zeta)-\varphi'\circ\varrho\circ\widetilde{\varphi}(\zeta),\, \overline{(\varphi^{\ast})'\circ\varrho\circ\widetilde{\varphi}(\zeta)}\rangle\big|\\
\leq&\,
\big|\langle
\widetilde{\varphi}'(\zeta)-\varphi'(\zeta),\, \overline{(\varphi^{\ast})'(\zeta)-(\varphi^{\ast})'\circ\varrho\circ\widetilde{\varphi}(\zeta)}
\rangle\big|\\
&+\big|\langle
\varphi'\circ\varrho\circ\widetilde{\varphi}(\zeta)-\varphi'(\zeta),\,
\overline{(\varphi^{\ast})'\circ\varrho\circ\widetilde{\varphi}(\zeta)}
\rangle\big|
\\
\leq&\, C_{\alpha}|\varrho\circ\widetilde{\varphi}(\zeta)-\zeta|^{\frac{2\alpha}{\alpha+1}}
\\
\leq&\, C_{\alpha} |\zeta-1|^{2\alpha}
\end{split}
\end{equation*}
for all $\zeta\in\Delta$. Now combining this with \eqref{ineq:estimate05} yields that
\begin{equation}\label{ineq:estimate06}
\big|II(\zeta)-\langle \widetilde{\varphi}'(\zeta)-\varphi'(\zeta),\, \overline{(\varphi^{\ast})'(\zeta)}\rangle\big|\leq C_{\alpha, \,\beta} |\zeta-1|^{2\alpha}
\end{equation}
for all $\zeta\in\mathcal{R}_{\beta}$. On the other hand, taking into account that
$$
\langle \widetilde{\varphi}',\, \overline{(\widetilde{\varphi}^{\ast})'}\rangle+
\langle \widetilde{\varphi}'',\, \overline{\widetilde{\varphi}^{\ast}}\rangle
=\langle \widetilde{\varphi}',\, \overline{\widetilde{\varphi}^{\ast}}\rangle'
=0
$$
on $\Delta$, we can rephrase \eqref{ineq:estimate03} as
\begin{equation}\label{ineq:estimate07}
\big|\langle \widetilde{\varphi}''(\zeta),\, \overline{({\rm gard }\,\varrho)\circ\widetilde{\varphi}(\zeta)}\rangle
+\langle \widetilde{\varphi}'(\zeta),\, \overline{(\widetilde{\varphi}^{\ast})'(\zeta)}\rangle\big|
\leq C_{\alpha, \,\beta}|\zeta-1|^{2\alpha}
\end{equation}
for all $\zeta\in\mathcal{R}_{\beta}$. Putting \eqref{eq:second der}, \eqref{ineq:estimate04}, \eqref{ineq:estimate06}, and \eqref{ineq:estimate07} together, we then conclude that for every $\alpha\in(0, 1)$ and every $\beta>1$,
$$
\big|
(\varrho\circ\widetilde{\varphi})''(\zeta)-\langle \widetilde{\varphi}'(\zeta)-\varphi'(\zeta),\, \overline{(\varphi^{\ast})'(\zeta)}\rangle-\langle \widetilde{\varphi}'(\zeta),\, \overline{(\varphi^{\ast})'(\zeta)-(\widetilde{\varphi}^{\ast})'(\zeta)}\rangle
\big|
\leq C_{\alpha, \,\beta}|\zeta-1|^{2\alpha}
$$
as $\zeta\in\mathcal{R}_{\beta}$. By symmetry, we also have for every $\alpha\in(0, 1)$ and every $\beta>1$,
$$
\big|
(\widetilde{\varrho}\circ\varphi)''(\zeta)-\langle\varphi'(\zeta)-\widetilde{\varphi}'(\zeta),\, \overline{(\widetilde{\varphi}^{\ast})'(\zeta)}\rangle-\langle\varphi'(\zeta),\, \overline{(\widetilde{\varphi}^{\ast})'(\zeta)-(\varphi^{\ast})'(\zeta)}\rangle
\big|
\leq C_{\alpha, \,\beta}|\zeta-1|^{2\alpha}
$$
as $\zeta\in\mathcal{R}_{\beta}$. Now adding the preceding two inequalities together yields that
\begin{equation}\label{ineq:key point of uniqueness01}
\big|(\varrho\circ\widetilde{\varphi}+\widetilde{\varrho}\circ\varphi)''(\zeta)
+2\langle\varphi'(\zeta)-\widetilde{\varphi}'(\zeta),\, \overline{(\varphi^{\ast})'(\zeta)-(\widetilde{\varphi}^{\ast})'(\zeta)}\rangle\big|\leq C_{\alpha, \,\beta}|\zeta-1|^{2\alpha}
\end{equation}
for all $\zeta\in\mathcal{R}_{\beta}$. Also, it is evident that
$$
\big|\langle\varphi'(\zeta)-\widetilde{\varphi}'(\zeta),\, \overline{(\varphi^{\ast})'(\zeta)-(\widetilde{\varphi}^{\ast})'(\zeta)}\rangle\big|\leq C_{\alpha}|\zeta-1|^{2\alpha}
$$
for all $\zeta\in\Delta$ and $\alpha\in(0, 1)$. Together with \eqref{ineq:key point of uniqueness01}, this further implies that for every $\alpha\in(0, 1)$ and every $\beta>1$,
$$
|(\varrho\circ\widetilde{\varphi}+\widetilde{\varrho}\circ\varphi)''(\zeta)|
\leq C_{\alpha, \,\beta}|\zeta-1|^{2\alpha}
$$
as $\zeta\in\mathcal{R}_{\beta}$. Now the desired claim \eqref{eq:key point of uniqueness} follows immediately.

By using the Cauchy integral formula, we can conclude from \eqref{eq:key point of uniqueness} that
$$
(\varrho\circ\widetilde{\varphi}+\widetilde{\varrho}\circ\varphi)'''(\zeta)\to 0
$$
as $\zeta\to 1$ non-tangentially. Then by Theorem \ref{thm:effectiveness of BK}, we have
\begin{equation}\label{eq:key point of uniqueness01}
\frac12(\varrho\circ\widetilde{\varphi}+\widetilde{\varrho}\circ\varphi)={\rm Id}_{\Delta}.
\end{equation}
We now prove that $\widetilde{\varphi}=\varphi$. Recall that $(\widetilde{\varphi}^{\ast})'(1)=(\varphi^{\ast})'(1)$ (see \eqref{eq:equi-der}), an argument completely analogous to the one at the beginning of the proof indicates that it is sufficient to show that $\widetilde{\varphi}(\overline{\Delta})=\varphi(\overline{\Delta})$. As usual, we argue by contradiction. If this were not the case, it would follow that
\begin{equation}\label{eq: key contradiction}
\widetilde{\varphi}(\overline{\Delta})\cap\varphi(\overline{\Delta})=\{p\}.
\end{equation}
On the other hand, it follows from Proposition \ref{prop:regularity of leftinv} that $\varrho(\overline{\Omega}\setminus\varphi(\partial\Delta))\subset\Delta$ and $\widetilde{\varrho}(\overline{\Omega}\setminus\widetilde{\varphi}(\partial\Delta))\subset\Delta$.
Now combining this with \eqref{eq: key contradiction}, we see that
$$
\frac12(\varrho\circ\widetilde{\varphi}+\widetilde{\varrho}\circ\varphi)(\overline{\Delta}\setminus\{1\})\subset\Delta,
$$
which contradicts \eqref{eq:key point of uniqueness01}. This completes the proof of the first statement part of the theorem. The second one follows easily from a similar argument as in the very beginning of the proof.
\end{proof}

\begin{remark}
Although the existence part of Theorem \ref{mainthm1} requires $\partial\Omega$ to be at least $C^3$-smooth as indicated by \cite{Huang-Pisa94}, the preceding argument implies that  for the uniqueness part: when $\Omega$ is strongly  convex, the $C^{2,\, \alpha}$-regularity of $\partial\Omega$ is enough  whenever $\alpha\in(1/2, 1)$; when $\Omega$ is only strongly linearly convex,  the $C^{2,\, \alpha}$-regularity  for $\partial\Omega$ is also enough whenever $\alpha\in((\sqrt 5-1)/2, 1)$. In the latter case, a slight modification of the preceding argument is necessary.
\end{remark}

\begin{remark}
The uniqueness result for complex geodesics of bounded strongly convex domains in $\Omega\subset\mathbb C^n\, (n>1)$ with $C^3$-smooth boundary was also stated in \cite[Lemma 2.7]{Gaussier-Seshadri}. However, the proof
given there seems to us to be  incorrect. In fact, by carefully checking that proof, one can see that what the authors of \cite{Gaussier-Seshadri} claimed is essentially the following (with the notation fixed there): Let $\Omega$ be  a  bounded strongly convex domain with $C^3$-smooth boundary and let $p\in \partial \Omega$.
Let $\phi$ be a complex geodesic of $\Omega$ with $\phi(1)=p$ and $\psi$ be a holomorphic mapping from $\Delta$ into $\Omega$ such that $\psi(1)=p$ and $\psi'(1)=\phi'(1)$. Then $\psi=\phi$, which is obviously not true. Even in the complex geodesic case for both mappings, as we proved in this paper, they are only the same after composing an automorphism. It seems to us that they overlooked the fact that the constant $C$
(in the proof of \cite[Lemma 2.7]{Gaussier-Seshadri}) goes to zero, instead of being uniformly bounded below by a positive constant, as the parameter $\eta\in\Delta$ tends to $1$. Indeed, just simply taking $\Omega=\Delta$, one can compute directly this constant and find out it goes to zero as $\eta\to 1$.
\end{remark}

\section{A new boundary spherical representation}\label{subsect:spherical rep}
Let $\Omega\subset\mathbb C^n\, (n>1)$ be a bounded strongly linearly convex domain with $C^3$-smooth boundary. Let $p\in\partial\Omega$ and $\nu_p$ be the unit outward normal  to $\partial\Omega$ at
$p$. Set
   $$
   L_p:=\big\{v\in\mathbb C^n\!: |v|=1,\ \langle v, \nu_p\rangle>0\big\}
   $$
and let $v\in L_p$. Then by  Theorem \ref{mainthm1}, we see that there exists a {\it unique} complex geodesic $\varphi_v$ of $\Omega$ such that $\varphi_v(1)=p$, $\varphi_v'(1)=\langle v, \nu_p\rangle v$ and
   $$
   \left.\frac{d}{d\theta}\right|_{\theta=0}|\varphi^{\ast}_v(e^{i\theta})|=0.
   $$
Here as before, $\varphi^{\ast}_v$ is the  dual mapping of $\varphi_v$. In what follows, we will refer to such a $\varphi_v$ as {\it the preferred complex geodesic of $\Omega$ associated to $v$}.

Up to a unitary transformation on $\mathbb C^n$, which does not change the strong  linear convexity of $\Omega$, we may assume that $\nu_p=e_1=(1,0,\ldots, 0)$ and thus $L_p$ is given by
   $$
   L_p=\big\{v\in\mathbb C^n\!: |v|=1,\ \langle v, e_1\rangle >0 \big\}.
   $$
Now it is easy to verify that for every $v\in L_p$, the mapping
   $$
   \eta_v\!:\Delta \ni \zeta \mapsto e_1+(\zeta-1)\langle v, e_1\rangle v
   $$
is the preferred complex geodesic of the open unit ball $\mathbb B^n\subset\mathbb C^n$ associated to $v$, since a straightforward calculation shows that
   $$
   \eta^{\ast}_v(\zeta)=\frac{\zeta e_1+(1-\zeta)\langle v, e_1\rangle \bar{v}}{\langle v, e_1\rangle^2}
   $$
and thus $|\eta^{\ast}_v|=1/\langle v, e_1\rangle^2$ on $\partial\Delta$. Also for every $z\in\overline{\Omega}\setminus\{p\}$, there exists a {\it unique} complex geodesic disc in $\Omega$ whose closure contains $z$ and $p$; see \cite[Theorem 1]{Chang-Hu-Lee88}. We can then appropriately parameterize this complex geodesic disc such that it is given by the image of the preferred complex geodesic $\varphi_{v_z}$ of $\Omega$ associated to a {\it unique} $v_z\in L_p$, in view of the proof of Theorem \ref{mainthm1} and Remark \ref{rmk:existence of geodesics}. This leads us to consider the mapping $\Psi_p\!:
\overline{\Omega}\to \overline{\mathbb B}^n$ defined by setting $\Psi_p(p)=e_1$, and
   $$
   \Psi_p(z)=e_1+(\zeta_z-1)\langle v_z, e_1\rangle v_z,\quad z\in\overline{\Omega}\setminus\{p\},
   $$
where $\zeta_z:=\varphi_{v_z}^{-1}(z)$. Clearly, $\Psi_p$ is a bijection with inverse $\Psi_p^{-1}$ given by $\Psi_p^{-1}(e_1)=p$, and
   $$
   \Psi_p^{-1}(w)=\varphi_{v_w}(\zeta_w),\quad w\in\overline{\mathbb B}^n\setminus\{e_1\},
   $$
where $(v_w,\,\zeta_w)\in L_p\times\overline{\Delta}$ is the unique data such that $\eta_{v_w}(\zeta_w)=w$; more explicitly,
   $$
   v_w=-\frac{1-\langle e_1, w\rangle}{|1-\langle e_1, w\rangle|}\frac{w-e_1}{|w-e_1|},
   \quad \zeta_w=1-\frac{|w-e_1|^2}{|1-\langle e_1, w\rangle|^2}(1-\langle w, e_1\rangle).
   $$
Moreover, we can prove the following

\begin{theorem}\label{thm:continuity}
Let $\Omega\subset\mathbb C^n\, (n>1)$ be a bounded strongly linearly convex domain with $C^3$-smooth boundary and let $p\in\partial\Omega$. Then
\begin{enumerate}[leftmargin=2.0pc, parsep=4pt]
\item  [{\rm (i)}]
    For every $\alpha\in (0, 1/2)$, the mapping
       $$
       L_p\ni v\mapsto \varphi_v\in C^{1,\,\alpha}(\overline{\Delta})
       $$
    is continuous, and so is
       $$
       L_p\ni v\mapsto \varphi^{\ast}_v\in C^{1,\,\alpha}(\overline{\Delta}).
       $$
\item [{\rm (ii)}]
    Both $\Psi_p$ and $\Psi_p^{-1}$ are continuous so that they are homeomorphisms.
\end{enumerate}
\end{theorem}

For the proof of the above theorem, we need the following
\begin{lemma}\label{lem:uniform extimate}
Let $\Omega\subset\mathbb C^n\, (n>1)$ be a bounded strongly linearly convex domain with $C^3$-smooth boundary, and $\mathcal{F}\subset \mathcal{O}(\Delta,\,\Omega)$ a family of complex geodesics of $\Omega$ such that
$$
\{\varphi(0)\!: \varphi\in\mathcal{F}\}\subset\subset\Omega.
$$
Then
$$
\sup_{\varphi\in\mathcal{F}}\big(\|\varphi\|_{C^{1,\,1/2}(\overline{\Delta})} +\|\varphi^{\ast}\|_{C^{1,\,1/2}(\overline{\Delta})}\big)<\infty,
$$
where for every $\varphi$, $\varphi^{\ast}$ denotes its dual mapping as before.
\end{lemma}

\begin{proof}
First of all, by \cite[Lemma 4]{Huang-Illinois94} (which is also valid for bounded strongly linearly convex domains in $\mathbb C^n$, in view of \cite{Lempert84, KW13}), we obtain that
\begin{equation}\label{ineq:uniform extimate1}
\sup_{\varphi\in\mathcal{F}}\|\varphi\|_{C^{1,\,1/2}(\overline{\Delta})}<\infty.
\end{equation}
So we are left to show that
\begin{equation*}
\sup_{\varphi\in\mathcal{F}}\|\varphi^{\ast}\|_{C^{1,\,1/2}(\overline{\Delta})}<\infty.
\end{equation*}
To this end, note that the standard proof of the Hardy-Littlewood theorem (see, e.g., \cite[Theorem 2.6.26]{Abatebook}) implies that the norms $\|\  \|_{C^{1,\,1/2}(\overline{\Delta})}$ and $\|\ \|_{C^{1,\,1/2}(\partial\Delta)}$ are equivalent on $\mathcal{O}(\Delta)\cap C^{1,\,1/2}(\overline{\Delta})$ (and even on ${\rm harm}(\Delta)\cap C^{1,\,1/2}(\overline{\Delta})$), we therefore need only show that
\begin{equation*}
\sup_{\varphi\in\mathcal{F}}\|\varphi^{\ast}\|_{C^{1,\,1/2}(\partial\Delta)}<\infty.
\end{equation*}
Furthermore, since $\varphi^{\ast}|_{\partial\Delta}(\zeta)=\zeta|\varphi^{\ast}(\zeta)|\overline{\nu\circ\varphi(\zeta)}$,  \eqref{ineq:uniform extimate1} can reduce the problem to
\begin{equation}\label{ineq:uniform extimate2}
\sup_{\varphi\in\mathcal{F}}\||\varphi^{\ast}|\|_{C^{1,\,1/2}(\partial\Delta)}<\infty.
\end{equation}

We follow an idea of Lempert \cite{Lempert81}. For every $\zeta_0\in\partial\Delta$ and every $\varphi\in\mathcal{F}$, we can first choose an integer $1\leq k_{\varphi,\, \zeta_0}\leq n$ such that
$$
|\langle e_{k_{\varphi,\, \zeta_0}},\, \nu\circ\varphi(\zeta_0)\rangle|\geq1\sqrt{n},
$$
and then by the equicontinuity of $\mathcal{F}$ (which follows easily from \eqref{ineq:uniform extimate1}) a small neighborhood $V_{\zeta_0}\subset\mathbb C$ (independent of $\varphi\in\mathcal{F}$) of $\zeta_0$ such that
$$
|\langle e_{k_{\varphi,\, \zeta_0}},\, \nu\circ\varphi(\zeta)\rangle|\geq1/2\sqrt{n}
$$
for all $\zeta\in \partial\Delta\cap V_{\zeta_0}$ and $\varphi\in\mathcal{F}$. We can further take
for every $\varphi\in\mathcal{F}$ a function $\chi_{\varphi,\, \zeta_0}\in C^{1,\,1/2}(\partial\Delta)$ such that
$$
\exp\circ\chi_{\varphi,\, \zeta_0}=\langle e_{k_{\varphi,\, \zeta_0}},\, \nu\circ\varphi\rangle
$$
on $\partial\Delta\cap V_{\zeta_0}$, and
\begin{equation}\label{ineq:equi-norm}
\|\chi_{\varphi,\, \zeta_0}\|_{C^{1,\,1/2}(\partial\Delta)}
\leq C_{n,\,\zeta_0}\|\nu\circ\varphi\|_{C^{1,\,1/2}(\partial\Delta)},
\end{equation}
where $C_{n,\,\zeta_0}>0$ is a constant depending only on $n$ and $\zeta_0$. Also, we can extend $-{\rm Im}\,\chi_{\varphi,\, \zeta_0}$ to a harmonic function in ${\rm harm}(\Delta)\cap C^{1,\,1/2}(\overline{\Delta})$, still denoted by $-{\rm Im}\,\chi_{\varphi,\, \zeta_0}$,  and let
$\rho_{\varphi,\, \zeta_0}\!:\Delta\to \mathbb R$ be its conjugate function such that $\rho_{\varphi,\, \zeta_0}(\zeta_0)=0$. Then by  the classical Privalov theorem (see, e.g., \cite[Proposition 6.2.10]{Baouendibook}), there exists a constant $C>0$ such that
\begin{equation}\label{ineq:Hilbert-bdd}
\|\rho_{\varphi,\, \zeta_0}\|_{C^{1,\,1/2}(\overline{\Delta})}\leq C\|\chi_{\varphi,\, \zeta_0}\|_{C^{1,\,1/2}(\partial \Delta)}
\end{equation}
for all $\varphi\in \mathcal{F}$. Now note that for every $\zeta\in \partial\Delta\cap V_{\zeta_0}$,
$$
|\varphi^{\ast}(\zeta)|\langle e_{k_{\varphi,\, \zeta_0}},\, \nu\circ\varphi(\zeta)\rangle
=\zeta^{-1}\langle e_{k_{\varphi,\, \zeta_0}},\, \overline{\varphi^{\ast}(\zeta)}\rangle,
$$
and
$$
\langle e_{k_{\varphi,\, \zeta_0}},\, \nu\circ\varphi(\zeta)\rangle \exp\circ(\rho_{\varphi,\, \zeta_0}-{\rm Re}\,\chi_{\varphi,\, \zeta_0})(\zeta)=\exp\circ(\rho_{\varphi,\, \zeta_0}+i{\rm Im}\,\chi_{\varphi,\, \zeta_0})(\zeta).
$$
This means that they can extend to holomorphic functions on $\Delta \cap V_{\zeta_0}$, then so does their quotient $|\varphi^{\ast}|\exp\circ({\rm Re}\,\chi_{\varphi,\, \zeta_0}-\rho_{\varphi,\, \zeta_0})$, which takes real values on $\partial\Delta\cap V_{\zeta_0}$, and hence can extend holomorphically across $\partial\Delta$. We denote by $f_{\varphi,\, \zeta_0}$ its holomorphic extension. Now taking into account that
$$
\sup_{\varphi\in\mathcal{F}}\|\varphi^{\ast}\|_{C(\overline{\Delta})}<\infty
$$
(see \cite{Lempert84, KW13}) and $\rho_{\varphi,\, \zeta_0}(\zeta_0)=0$, we can assume that $\{f_{\varphi,\, \zeta_0}\}_{\varphi\in\mathcal{F}}$ is uniformly bounded  by shrinking $V_{\zeta_0}$ {\it uniformly}, in view of \eqref{ineq:uniform extimate1}, \eqref{ineq:equi-norm} and \eqref{ineq:Hilbert-bdd}. Moreover, by shrinking $V_{\zeta_0}$ again, the classical Cauchy estimate allows us to conclude that
$$
\sup_{\varphi\in\mathcal{F}}\|f_{\varphi,\, \zeta_0}\|_{C^{1,\,1/2}(\partial\Delta \cap V_{\zeta_0})}< \infty.
$$
Together with \eqref{ineq:uniform extimate1}, \eqref{ineq:equi-norm} and \eqref{ineq:Hilbert-bdd}, this further implies that
$$
\sup_{\varphi\in\mathcal{F}}\||\varphi^{\ast}|\|_{C^{1,\,1/2}(\partial\Delta \cap V_{\zeta_0})}<\infty,
$$
since
$$
|\varphi^{\ast}|=f_{\varphi,\, \zeta_0}\exp\circ(\rho_{\varphi,\, \zeta_0}-{\rm Re}\,\chi_{\varphi,\, \zeta_0})
$$
on $\partial\Delta \cap V_{\zeta_0}$. Now covering $\partial \Delta$ by finitely many open sets in $\mathbb C$ (like $V_{\zeta_0}$), the Lebesgue number lemma gives \eqref{ineq:uniform extimate2}. This completes the proof.
\end{proof}

We now can prove Theorem $\ref{thm:continuity}$.

\begin{proof}[Proof of Theorem $\ref{thm:continuity}$]
Up to a unitary transformation on $\mathbb C^n$, we may assume that $\nu_p=e_1$.

{\rm (i)} Fix $\alpha\in (0, 1/2)$. Let $v_0\in L_p$ and $\{v_k\}_{k\in\mathbb N}\subset L_p$ be a sequence converging to $v_0$. It suffices to show that  $\{\varphi_{v_k}\}_{k\in\mathbb N}$ converges to $\varphi_{v_0}$ in the topology of $C^{1,\,\alpha}(\overline{\Delta})$, and $\{\varphi^{\ast}_{v_k}\}_{k\in\mathbb N}$ converges to $\varphi^{\ast}_{v_0}$ in the same topology.

First of all, since
   $$
   |\langle\varphi'_{v_k}(1),\, e_1\rangle|/|\varphi'_{v_k}(1)|=\langle v_k, e_1\rangle
   \rightarrow \langle v_0, e_1\rangle>0
   $$
as $k\to \infty$, it follows from \cite[Theorem 2]{Huang-Pisa94} that
\begin{equation}\label{ineq:diameter1}
   \inf_{k\in\mathbb N}{\rm diam} \, \varphi_{v_k} (\Delta)>0,
\end{equation}
where ${\rm diam} \, \varphi_{v_k} (\Delta)$ denotes the Euclidean diameter of $\varphi_{v_k} (\Delta)$.

\medskip
\noindent  {\bf Claim :}  $\{\varphi_{v_k}(0)\!: k\in\mathbb N\}\subset\subset\Omega$.
\medskip

Seeking a contradiction, suppose not. Then by passing to a subsequence if necessary, we may assume that
\begin{equation}\label{eq: contradiction}
   \varphi_{v_k}(0)\rightarrow \partial\Omega \quad {\rm as} \,\ k\to \infty.
\end{equation}
For  $k\in\mathbb N$, let $\zeta_k\in\Delta$ be such that $\varphi_{v_k}\circ\sigma_k$ is a normalized complex geodesic of $\Omega$, i.e.,
   $$
   {\rm dist}(\varphi_{v_k}\circ\sigma_k(0),\,\partial\Omega)
   =\max_{\zeta\in\overline{\Delta}}{\rm dist}(\varphi_{v_k}\circ\sigma_k(\zeta),\,\partial\Omega),
   $$
where $ {\rm dist}(\,\cdot\,,\partial\Omega)$ denotes the Euclidean distance  to the boundary $\partial\Omega$, and
\begin{equation}\label{eq:new paramter}
   \sigma_k(\zeta):=\frac{1-\overline{\zeta}_k}{1-\zeta_k}
   \frac{\zeta-\zeta_k}{1-\overline{\zeta}_k \zeta}\in{\rm{Aut}}(\Delta).
\end{equation}
Then by \cite[Proposition 4]{Chang-Hu-Lee88},
   $$
   \inf_{k\in\mathbb N}{\rm dist}(\varphi_{v_k}\circ\sigma_k(0),\,\partial\Omega)>0.
   $$
Thus by Lemma \ref{lem:uniform extimate}, we conclude that both $\{\varphi_{v_k}\circ\sigma_k\}_{k\in\mathbb N}$ and $\{(\varphi_{v_k}\circ\sigma_k)^{\ast}\}_{k\in\mathbb N}$ satisfy a uniform  $C^{1,\, 1/2}$-estimate. Now in light of the classical Ascoli-Arzel\`{a} theorem,  we may assume, without loss of generality,  that these two sequences converge to $\varphi_{\infty}$, $\widetilde{\varphi}_{\infty}\in
\mathcal{O}(\Delta)\cap C^{1,\, 1/2}(\overline{\Delta})$, respectively, in the topology of $C^{1,\,\alpha}(\overline{\Delta})$. Note that $\Omega$ is strongly pseudoconvex and
   $$
   {\rm diam} \, \varphi_{\infty}(\Delta)
   \geq\inf_{k\in\mathbb N}{\rm diam} \, \varphi_{v_k} (\Delta)>0
   $$
by (\ref{ineq:diameter1}), we see that $\varphi_{\infty}(\Delta)\subset\Omega$ and then by the continuity of
the Kobayashi distance, $\varphi_{\infty}$ is a complex geodesic of $\Omega$. Clearly, $\varphi_{\infty}(1)=p$ and
   $$
   \varphi'_{\infty}(1)=\lim_{k\rightarrow \infty}(\varphi_{v_k}\circ\sigma_k)'(1)
   =\lim_{k\rightarrow \infty}\sigma_k'(1)\langle v_k, e_1\rangle v_k.
   $$
Note also that $\langle v_k, e_1\rangle v_k \to \langle v_0, e_1\rangle v_0\in \mathbb C^n\setminus\{0\}$, we deduce that $\lim_{k\to\infty}\sigma_k'(1)$ exists. Moreover, since $\sigma_k'(1)>0$, it follows from the Hopf lemma (see, e.g., Remark \ref{rmk:existence of geodesics}) that
\begin{equation}\label{eq:estimate of der1}
   \lim_{k\to\infty}\sigma_k'(1)=|\varphi'_{\infty}(1)|/\langle v_0, e_1\rangle\in (0,\infty).
\end{equation}

We now proceed to show that there exists an $\varepsilon\in (0, 1)$ such that
\begin{equation}\label{eq:inclusion}
   \big\{\overline{\zeta}_k\!: k\in\mathbb N\big\}\subset
   \left\{\zeta\in\Delta\!: \varepsilon<\frac{|\zeta-1|^2}{1-|\zeta|^2}<\frac1{\varepsilon}, \
   |\zeta-1|>\varepsilon
   \right\}.
\end{equation}
In particular, this implies that the sequence $\{\sigma_k\}_{k\in\mathbb N}$ is relatively compact in ${\rm{Aut}}(\Delta)$ with respect to the compact-open topology so that we may assume that it  converges to some $\sigma_{\infty}\in{\rm{Aut}}(\Delta)$. Consequently, we see that
   $$
   \varphi_{v_k}(0)=\varphi_{v_k}\circ\sigma_k\circ\sigma^{-1}_k(0)
   \to\varphi_{\infty}\circ\sigma^{-1}_{\infty}(0)\in \Omega
   $$
as $k\to\infty$. This contradicts (\ref{eq: contradiction}), and thus the preceding claim follows.

To show the existence of $\varepsilon\in (0, 1)$ satisfying (\ref{eq:inclusion}), we make use of the fact that all $\varphi_{v_k}$'s are preferred. By the definition of dual mappings, we have
$$
\sigma'_k(\varphi_{v_k}\circ\sigma_k)^{\ast}=\varphi^{\ast}_{v_k}\circ\sigma_k
$$
on $\overline \Delta$. It then follows that
   \begin{equation*} \label{1-1}
   |\sigma'_k(1)|\left.\frac{d}{d\theta}\right|_{\theta=0}|
   (\varphi_{v_k}\circ\sigma_k)^{\ast}(e^{i\theta})|
   +\frac{|\varphi^{\ast}_{v_k}(1)|}{|\sigma'_k(1)|}
   \left.\frac{d}{d\theta}\right|_{\theta=0}|\sigma'_k(e^{i\theta})|
   =\left.\frac{d}{d\theta}\right|_{\theta=0}|\varphi^{\ast}_{v_k}\circ\sigma_k(e^{i\theta})|.
   \end{equation*}
Note that $|\varphi^{\ast}_{v_k}(1)|=1/\langle v_k, e_1\rangle^2$,
   $$
   \left.\frac{d}{d\theta}\right|_{\theta=0}|\varphi^{\ast}_{v_k}\circ\sigma_k(e^{i\theta})|
   =|\sigma'_k(1)|\left.\frac{d}{d\theta}\right|_{\theta=0}|\varphi^{\ast}_{v_k}(e^{i\theta})|
   =0,
   $$
and recall also that $\{(\varphi_{v_k}\circ\sigma_k)^{\ast}\}_{k\in\mathbb N}$ converges
to $\widetilde{\varphi}_{\infty}$ in $C^1(\overline{\Delta})$ as $k\to\infty$. We then conclude that
\begin{equation}\label{eq:estimate of der2}
   \begin{split}
      \left.\frac{d}{d\theta}\right|_{\theta=0}|\sigma'_k(e^{i\theta})|
      &=-\langle v_k, e_1\rangle^2|\sigma'_k(1)|^2
      \left.\frac{d}{d\theta}\right|_{\theta=0}|(\varphi_{v_k}\circ\sigma_k)^{\ast}(e^{i\theta})|\\
      &\to -|\varphi'_{\infty}(1)|^2
      \left.\frac{d}{d\theta}\right|_{\theta=0}|\widetilde{\varphi}_{\infty}(e^{i\theta})|.
   \end{split}
\end{equation}
Now in view of  \eqref{eq:estimate of der1},  \eqref{eq:estimate of der2}, and using the explicit formula \eqref{eq:new paramter},  we can easily find an $\varepsilon\in (0, 1)$ such that  \eqref{eq:inclusion} holds.

\smallskip
Now we are ready to check the desired continuity. This part is very similar to the proof of the preceding claim. Indeed, it follows first from Lemma \ref{lem:uniform extimate} that both $\{\varphi_{v_k}\}_{k\in\mathbb N}$ and $\{\varphi^{\ast}_{v_k}\}_{k\in\mathbb N}$ satisfy a uniform
$C^{1,\, 1/2}$-estimate so that the set of limit points of $\{\varphi_{v_k}\}_{k\in\mathbb N}$ in the topology of $C^{1,\,\alpha}(\overline{\Delta})$ is non-empty, and the same is true for the sequence $\{\varphi_{v_k}^{\ast}\}_{k\in\mathbb N}$. Therefore, we need only show that $(\varphi_{v_0},\,
\varphi^{\ast}_{v_0})$ is the only limit point of $\{(\varphi_{v_k},\, \varphi^{\ast}_{v_k})\}_{k\in\mathbb N}$ in $C^{1,\,\alpha}(\overline{\Delta})\times C^{1,\,\alpha}(\overline{\Delta})$. Without loss of generality, we assume that the sequence $\{(\varphi_{v_k},\, \varphi^{\ast}_{v_k})\}_{k\in\mathbb N}$ itself converges to $(\varphi_{\infty}, \, \widetilde{\varphi}_{\infty})$. Then as before, we see that $\varphi_{\infty}$ is a complex geodesic of  $\Omega$ and $\widetilde{\varphi}_{\infty}={\varphi}^{\ast}_{\infty}$. Moreover,
$\varphi_{\infty}(1)=p$, $\varphi'_{\infty}(1)=\langle v_0, e_1\rangle v_0$, and
   $$
   \left.\frac{d}{d\theta}\right|_{\theta=0}|\varphi^{\ast}_{\infty}(e^{i\theta})|
   =\left.\frac{d}{d\theta}\right|_{\theta=0}|\widetilde{\varphi}_{\infty}(e^{i\theta})|
   =\lim_{k\to\infty}\left.\frac{d}{d\theta}\right|_{\theta=0}|\varphi^{\ast}_{v_k}(e^{i\theta})|
   =0.
   $$
Now by uniqueness (see Theorem \ref{mainthm1}), we see that $\varphi_{\infty}=\varphi_{v_0}$ and then $\widetilde{\varphi}_{\infty}=\varphi^{\ast}_{\infty}=\varphi^{\ast}_{v_0}$ as desired.

\medskip
{\rm (ii)}
We only check the continuity of $\Psi_p$, since the continuity of $\Psi_p^{-1}$ can be verified in an analogous way, or alternatively follows immediately by using the well-known fact that an injective continuous mapping from a compact topological space to a Hausdorff space is necessarily an embedding and noting that $\Psi_p\!: \overline{\Omega}\to \overline{\mathbb B}^n$ is such a mapping.

Let $z_0\in\overline{\Omega}$ and $\{z_k\}_{k\in\mathbb N}$ be a sequence in $\overline{\Omega}\setminus\{p\}$ converging to $z_0$. For every $k\in\mathbb N$, let $(v_{z_k},\,\zeta_{z_k})\in L_p\times (\overline{\Delta}\setminus\{1\})$ be the unique data such that
$\varphi_{v_{z_k}}(\zeta_{z_k})=z_k$, where $\varphi_{v_{z_k}}$ is the preferred complex geodesic of $\Omega$ associated to $v_{z_k}$. We then need to consider the following two cases:

\medskip
\noindent  {\bf Case 1:} $z_0=p$.
\smallskip

It suffices to show that
   $$
   \lim_{k\to \infty}|(\zeta_{z_k}-1)\langle v_{z_k}, e_1\rangle|=0.
   $$
Suppose on the contrary that this is not the case. Then by passing to a subsequence, we may assume that
   $$
   (v_{z_k},\, \zeta_{z_k})\rightarrow(v_{\infty}, \,\zeta_{\infty})
   \in L_p\times (\overline{\Delta}\setminus\{1\})
   $$
as $k\to \infty$. Then by {\rm (i)}, we have $\varphi_{v_{\infty}}(\zeta_{\infty})=p$.  Thus by the injectivity of $\varphi_{v_{\infty}}$ on $\overline{\Delta}$,  we see that $\zeta_{\infty}=1$, giving a contradiction.

\medskip
\noindent {\bf Case 2:} $z_0\in\overline{\Omega}\setminus\{p\}$.
\smallskip

By definition, it suffices to show that
   $$
   \lim_{k\to \infty}(v_{z_k},\,\zeta_{z_k})=(v_{z_0},\, \zeta_{z_0}),
   $$
where $(v_{z_0},\,\zeta_{z_0})\in L_p\times \overline{\Delta}$ is the unique data such that $\varphi_{v_{z_0}}(\zeta_{z_0})=z_0$, where $\varphi_{v_{z_0}}$ is  the preferred complex geodesic of $\Omega$ associated to $v_{z_0}$. In other words, $(v_{z_0},\,\zeta_{z_0})$ is the only limit point of the sequence $\{(v_{z_k},\,\zeta_{z_k})\}_{k\in\mathbb N}$.

By the compactness of $\partial\mathbb B^n\times\overline{\Delta}$, we see that the set of limit points of $\{(v_{z_k},\, \zeta_{z_k})\}_{k\in\mathbb N}$ is non-empty. Therefore, without loss of generality, we may assume that $\{(v_{z_k},\,\zeta_{z_k})\}_{k\in\mathbb N}$ itself converges to some $(v_{\infty},\, \zeta_{\infty})\in \partial\mathbb B^n\times\overline{\Delta}$. Then it remains to show that
$v_{\infty}=v_{z_0}$ and $\zeta_{\infty}=\zeta_{v_0}$. To this end, note first that ${\rm diam} \, \varphi_{v_{z_k}} (\Delta)\geq|z_k-p|$ and $z_k\rightarrow z_0\in\overline{\Omega}\setminus\{p\}$, we see
that
\begin{equation}\label{ineq:diameter2}
   \inf_{k\in\mathbb N}{\rm diam} \, \varphi_{v_{z_k}} (\Delta)>0.
\end{equation}
Now we claim that $v_{\infty}\in L_p$. Indeed, if this were not the case,  it would hold that
$\langle v_{\infty}, e_1\rangle=0$, i.e., $v_{\infty}\in T_p^{1,\, 0}\partial\Omega$.
Therefore,
   $$
   |\langle\varphi'_{v_{z_k}}(1),\,e_1\rangle|/|\varphi'_{v_{z_k}}(1)|
   =\langle v_{z_k}, e_1\rangle \rightarrow 0
   $$
as $k\to \infty$. Thus by a preservation principle for extremal mappings (see Theorem 1 or Corollary 1 in \cite{Huang-Illinois94}), it follows that ${\rm diam}\,\varphi_{v_{z_k}} (\Delta)\rightarrow 0$ as
$k\to \infty$. This contradicts inequality \eqref{ineq:diameter2}.

Now by ${\rm (i)}$ again, we see that $\varphi_{v_{\infty}}(\zeta_{\infty})=z_0=\varphi_{v_{z_0}}(\zeta_{z_0})$
and $\varphi_{v_{\infty}}(1)=p=\varphi_{v_{z_0}}(1)$. Then by uniqueness
(see \cite[pp. 362--363]{Lempert84}), there exists a $\sigma\in {\rm{Aut}}(\Delta)$ such that
$\varphi_{v_{\infty}}=\varphi_{v_{z_0}}\circ\sigma$. Moreover, by the injectivity of $\varphi_{v_{z_0}}$ on $\overline{\Delta}$, it follows that $\sigma(\zeta_{\infty})=\zeta_{z_0}$ and $\sigma(1)=1$
(and hence $\sigma'(1)>0$). Note also that $\varphi'_{v_{\infty}}(1)=\sigma'(1)\varphi'_{v_{z_0}}(1)$ and
$|v_{\infty}|=|v_{z_0}|=1$, we deduce that $\sigma'(1)=1$ and $v_{\infty}=v_{z_0}$. Consequently, $\sigma$ is the identity by uniqueness (see Theorem \ref{mainthm1}) and thus $\zeta_{\infty}=\zeta_{v_0}$ as
desired.

Now the proof is complete.
\end{proof}

Let $\Omega\subset\mathbb C^n\, (n>1)$ be as described in Theorem \ref{thm:continuity}. To indicate the definition of $\Psi_p$ depends on the base point $p\in\partial\Omega$, we rewrite $\Psi_p(p)=\nu_p$, and
   $$
   \Psi_p(z)=\nu_p+(\zeta_{z,\,p}-1)\langle v_{z,\,p},\, \nu_p\rangle v_{z,\,p},\quad
   z\in\overline{\Omega}\setminus\{p\},
   $$
where $\zeta_{z,\,p}:=\varphi_{v_{z,\,p}}^{-1}(z)$, and $v_{z,\,p}\in L_p$ is the unique data such that the associated preferred complex geodesic $\varphi_{v_{z,\,p}}$ (with base point $p$, i.e., $\varphi_{v_{z,\,p}}(1)=p$) passes through $z$, i.e.,  $z\in \varphi_{v_{z,\,p}}(\overline{\Delta})$. Then we can prove the following

\begin{theorem}\label{thm:joint-continuity}
Let $\Omega\subset\mathbb C^n\, (n>1)$ be a bounded strongly linearly convex domain with $C^3$-smooth boundary. Then
\begin{enumerate}[leftmargin=2.0pc, parsep=4pt]
\item [{\rm (i)}]
    The mapping
       $$
       \partial\Omega\ni p\mapsto \Psi_p\in C(\overline{\Omega})
       $$
    is continuous.
\item [{\rm (ii)}]
    The mapping $\Psi\!:\overline{\Omega}\times\partial\Omega\to\overline{\mathbb B}^n$ given by
       $$
       \Psi(z, p)=\Psi_p(z)
       $$
    is continuous.
\end{enumerate}
\end{theorem}

\begin{proof}
The proof is essentially the same as that of Theorem \ref{thm:continuity}.

{\rm (i)} Suppose that this is not the case. Then we can find a sequence $\{(z_k,\, p_k)\}_{k\in\mathbb N}$ converging to some point  $(z_0,\, p_0)\in\overline{\Omega}\times\partial\Omega$ such that
   $$
   \inf_{k\in\mathbb N}|\Psi_{p_k}(z_k)-\Psi_{p_0}(z_k)|>0.
   $$
By the continuity of $\Psi_{p_0}$, we can find a $k_0\in\mathbb N$ such that
\begin{equation}\label{ineq:discontinuous}
\inf_{k\geq k_0}|\Psi_{p_k}(z_k)-\Psi_{p_0}(z_0)|>0.
\end{equation}
For every $k\in\mathbb N$, let $(v_{z_k,\,p_k},\, \zeta_{z_k,\,p_k})\in L_{p_k}\times
 \overline{\Delta}$ be the unique data such that
   $$
   \varphi_{v_{z_k,\,p_k}}(\zeta_{z_k,\,p_k})=z_k,
   $$
where $\varphi_{v_{z_k,\,p_k}}$ is  the preferred complex geodesic of $\Omega$
associated to $v_{z_k,\,p_k}$ (with base point $p_k$). The remaining argument
 is divided into the following two cases:

\medskip
\noindent  {\bf Case 1:} $z_0=p_0$.
\smallskip

Since $\nu_{p_k}\to \nu_{p_0}$, with $k_0$ replaced by a larger integer, we may assume that
   $$
   \inf_{k\geq k_0}|(\zeta_{z_k,\,p_k}-1)\langle v_{z_k,\,p_k},\,\nu_{p_k}\rangle|>0.
   $$
Then by passing to a subsequence, we may assume that
   $$
   (v_{z_k,\,p_k},\, \zeta_{z_k,\,p_k})\rightarrow(v_{\infty}, \, \zeta_{\infty})
   \in L_{p_0}\times (\overline{\Delta}\setminus\{1\})
   $$
as $k\to \infty$. Thus it follows that
   $$
   |\langle\varphi'_{v_{z_k,\,p_k}}(1),\, \nu_{p_k}\rangle|/|\varphi'_{v_{z_k,\,p_k}}(1)|
   =\langle v_{z_k,\,p_k},\, \nu_{p_k}\rangle
   \rightarrow \langle v_{\infty},\, \nu_{p_0}\rangle>0
   $$
as $k\to \infty$. Together with \cite[Theorem 2]{Huang-Pisa94}, this further implies that
   $$
   \inf_{k\in\mathbb N}{\rm diam} \, \varphi_{v_{z_k,\,p_k}} (\Delta)>0.
   $$
Then a same argument as in the proof of Theorem \ref{thm:continuity} (i) shows that
$$
\{\varphi_{v_{z_k,\,p_k}}(0)\!: k\in\mathbb N\}\subset\subset\Omega,
$$
and consequently, it follows that $\{\varphi_{v_{z_k,\,p_k}}\}_{k\in\mathbb N}$ satisfies a
 uniform  $C^{1/2}$-estimate; see \cite[Proposition 8]{Lempert84} and also
 \cite[Proposition 1.6]{Huang-Canad95}. Therefore, we may further assume that
  $\{\varphi_{v_{z_k,\,p_k}}\}_{k\in\mathbb N}$ itself converges uniformly on
  $\overline{\Delta}$ to a complex geodesic $\varphi_{\infty}$ of $\Omega$. Clearly,
  $\varphi_{\infty}(1)=p_0$. On the other hand, taking into account that
  $\varphi_{v_{z_k,\,p_k}}(\zeta_{z_k,\,p_k})=z_k$ and letting $k\rightarrow
   \infty$ yield that $\varphi_{\infty}(\zeta_{\infty})=p_0$.
    Then by the injectivity of $\varphi_{\infty}$ on $\overline{\Delta}$,
    we see that $\zeta_{\infty}=1$, giving a contradiction.

\medskip
\noindent  {\bf Case 2:} $z_0\in\overline{\Omega}\setminus\{p_0\}$.
\smallskip

Obviously, we can assume that $z_k\neq p_k$ for all $k\in\mathbb N$. Also, we may assume that $\{(v_{z_k,\,p_k},\, \zeta_{z_k,\,p_k})\}_{k\in\mathbb N}$ itself converges to some $(v_{\infty}, \, \zeta_{\infty})\in \partial\mathbb B^n\times\overline{\Delta}$. We shall show that
$v_{\infty}=v_{z_0,\,p_0}$ and $\zeta_{\infty}=\zeta_{z_0,\,p_0}$, where $(v_{z_0,\,p_0},\, \zeta_{z_0,\,p_0})\in L_{p_0}\times \overline{\Delta}$ is the unique data such that
$\varphi_{v_{z_0,\,p_0}}(\zeta_{z_0,\,p_0})=z_0$, where $\varphi_{v_{z_0,\,p_0}}$ is  the preferred complex geodesic of $\Omega$ associated to $v_{z_0,\,p_0}$ (with base point $p_0$). Clearly, it will follow that the sequence $\{\Psi_{p_k}(z_k)\}_{k\in\mathbb N}$ converges to $\Psi_{p_0}(z_0)$. This will contradict inequality \eqref{ineq:discontinuous}.

Arguing as in Case 2 in the proof of Theorem \ref{thm:continuity} (ii), we see that
   $$
   \inf_{k\in\mathbb N}{\rm diam} \, \varphi_{v_{z_k,\,p_k}} (\Delta)>0,
   $$
and thus $v_{\infty}\in L_{p_0}$. Then, we can argue again as in the proof of Theorem
\ref{thm:continuity} (i) to conclude that
   $$
   \{\varphi_{v_{z_k,\,p_k}}(0)\!: k\in\mathbb N\}\subset\subset\Omega.
   $$
This in turn implies that both $\{\varphi_{v_{z_k,\,p_k}}\}_{k\in\mathbb N}$ and $\{\varphi^{\ast}_{v_{z_k,\,p_k}}\}_{k\in\mathbb N}$ satisfy a uniform  $C^{1,\, 1/2}$-estimate. Consequently, we may assume that $\{(\varphi_{v_{z_k,\,p_k}},\,\varphi_{v_{z_k,\,p_k}}^{\ast})\}_{k\in\mathbb
N}$ converges in $C^1(\overline{\Delta})\times C^1(\overline{\Delta})$ to $(\varphi_{\infty},\, \varphi_{\infty}^{\ast})$, where $\varphi_{\infty}$ is a complex geodesic of $\Omega$ with
$\varphi_{\infty}^{\ast}$ as its dual mapping. Clearly, $\varphi_{\infty}(1)=p_0$, $\varphi_{\infty}(\zeta_{\infty})=z_0$ and $\varphi'_{\infty}(1)=\langle v_{\infty},\, \nu_{p_0}\rangle
v_{\infty}$, as well as
  $$
  \left.\frac{d}{d\theta}\right|_{\theta=0}|\varphi^{\ast}_{\infty}(e^{i\theta})|=0.
  $$
Then by uniqueness, $\varphi_{\infty}=\varphi_{v_{z_0,\,p_0}}$ and $v_{\infty}=v_{z_0,\,p_0}$, $\zeta_{\infty}=\zeta_{z_0,\,p_0}$ as desired.

{\rm (ii)} Follows immediately from {\rm (i)} together with Theorem \ref{thm:continuity} (ii).
\end{proof}

\section{Properties of the new boundary spherical representation}\label{sect:CMA equations}
In this section we prove our second main result (Theorem \ref{mainthm2}). To this end, we further investigate the properties of the boundary spherical representation that we constructed in the preceding section.

\subsection{Preservation of horospheres and non-tangential approach regions}\label{props-BSR}
Let $\Omega\subset\mathbb C^n\, (n>1)$ be a bounded strongly linearly convex domain with $C^3$-smooth boundary, and let $p\in\partial\Omega$. Recall first that according to Abate \cite{Abatehoro}, a {\it horosphere} $E_{\Omega}(p, z_0, R)$ of center $p\in \partial\Omega$, pole $z_0\in \Omega$ and radius $R>0$ is defined as
$$
E_{\Omega}(p, z_0, R):=\left\{
z\in \Omega\!: \lim_{w\to p}\big(k_{\Omega}(z, w)-k_{\Omega}(z_0, w)\big)<\frac{1}{2}\log R
\right\},
$$
where $k_{\Omega}$ denotes the Kobayashi distance on $\Omega$. When $\Omega$ is strongly convex, the existence of the limit in the definition of horospheres is well-known; see, e.g., \cite[Theorem 2.6.47]{Abatebook}. It is also the case when $\Omega$ is only strongly linearly convex, since it follows from the work of Lempert \cite{Lempert84, Lempert86}, Guan \cite{GuanMAeqn}, and Blocki \cite{Blocki00, Blocki01} that the pluricomplex Green function
\begin{equation}\label{Greenkernel}
   g_{\Omega}(\,\cdot\,,w)=\log\tanh k_{\Omega}(\,\cdot\,, w)
   \in C^{1,\, 1}(\overline{\Omega}\setminus\{w\})\,
   \footnote{That $g_{\Omega}(\,\cdot\,,w)\in C^1(\overline{\Omega}\setminus\{w\})$ is enough for our purpose here.}
\end{equation}
for all $w\in\Omega$, so that the proof of \cite[Theorem 2.6.47]{Abatebook} can be easily modified and thus applies. When $\Omega=\mathbb B^n$, the open unit ball in $\mathbb C^n$, an easy calculation using the explicit formula for $k_{\mathbb B^n}$  shows that
\begin{equation}\label{eq:horosphere in ball}
   E_{\mathbb B^n}(p, 0, R)=
   \left\{z\in \mathbb B^n\!: \frac{|1-\langle z, p\rangle|^2}{1-|z|^2}<R \right\};
\end{equation}
see, e.g., \cite[Section 2.2.2]{Abatebook}. Geometrically, it is an ellipsoid of the Euclidean center $c:=p/(1+R)$, its intersection with the complex plane $\mathbb Cp$ is a Euclidean disc of radius $r:=R/(1+R)$, and its intersection with the affine subspace through $c$ orthogonal to $\mathbb Cp$ is a Euclidean ball of the larger radius $\sqrt{r}$.

For our later purpose, we also need the following
\begin{proposition}\label{prop:basic facts}
Let $\Omega\subset\mathbb C^n\, (n>1)$ be a bounded strongly linearly convex domain  with $C^3$-smooth boundary. Let $\varphi$ be a complex geodesic of $\Omega$, and  $\rho\in \mathcal{O}(\Omega,\, \Omega)$ the Lempert retract associated with $\varphi$.  Then
\begin{enumerate}[leftmargin=2.0pc, parsep=4pt]
\item [{\rm (i)}]
    For every $(\zeta, v)\in\partial\Delta\times\mathbb C^n$, one has
       $$
       d(\varphi^{-1}\circ\rho)_{\varphi(\zeta)}(v)
       =\frac{\langle v,\,\nu\circ\varphi(\zeta)\rangle}
       {\langle \varphi'(\zeta),\, \nu\circ\varphi(\zeta)\rangle},
       $$
    where $\nu$ denotes the unit  outward normal vector field of $\partial\Omega$.
\item  [{\rm (ii)}]
    For every $p\in\varphi(\overline{\Delta})\cap\partial\Omega$ and every non-tangential
    continuous curve $\gamma\!: [0,1)\to\Omega$ terminating at $p$, one has
       $$
       \lim_{t\to 1^-}k_{\Omega}(\gamma(t),\, \rho\circ\gamma(t))=0.
       $$
\end{enumerate}
\end{proposition}

\begin{proof}
{\rm (i)} Set $\varrho:=\varphi^{-1}\circ\rho$. Then $\varrho$ is the so-called Lempert left inverse of $\varphi$ (see the paragraph proceeding Proposition \ref{prop:regularity of leftinv}), and $\varrho\in \mathcal{O}(\Omega,\, \Delta)\cap C^1(\overline{\Omega})$. Moreover,  equality \eqref{eq: re-dual-lefinv} gives
$$
\varphi^{\ast}=({\rm gard }\,\varrho)\circ\varphi=\frac{\partial
\varrho}{\partial z}\circ\varphi
$$
on $\overline{\Delta}$. Note also that
$$
\varphi^{\ast}|_{\partial\Delta}(\zeta)
=\frac{\overline{\nu\circ\varphi(\zeta)}}{\langle \varphi'(\zeta),\,
\nu\circ\varphi(\zeta)\rangle},
$$
the desired result follows immediately.

{\rm (ii)}
First of all, we can argue as in the proof of \cite[Lemma 2.7.12 (iii)]{Abatebook} to conclude that
$$
\lim_{t\to 1^-}\frac{|\gamma(t)-\rho\circ\gamma(t)|^2}
{{\rm dist}(\rho\circ\gamma(t),\, \partial\Omega)}=0.
$$
The remaining argument is the same as the proof of the second part of \cite[Proposition 2.7.11]{Abatebook}, and we leave the details to the interested reader.
\end{proof}

Now we prove the following result. The proof of the first part is analogous to that of \cite[Proposition 6.1]{Bracci-MathAnn}. We provide a detailed proof by modifying the argument given there, with the help of Proposition \ref{prop:basic facts}. The proof of the second part relies heavily on \cite[Theorem 2]{Huang-Pisa94} and Theorem \ref{thm:continuity}.

\begin{proposition}\label{prop:horo-preserving}
Let $\Omega\subset\mathbb C^n\, (n>1)$ be a bounded strongly linearly convex domain with $C^3$-smooth boundary. Let $p\in\partial\Omega$ and $\Psi_p\!: \overline{\Omega}\to \overline{\mathbb B}^n$ be the boundary spherical representation given in Section $\ref{subsect:spherical rep}$. Then
\begin{enumerate}[leftmargin=2.0pc, parsep=4pt]
\item [{\rm (i)}]
   For every $z_0\in\Omega$ and every $R>0$, one has
   $$
   \Psi_p(E_{\Omega}(p, z_0, R))=E_{\mathbb B^n}(\nu_p,\, \Psi_p(z_0),\, R).
   $$

{
\item  [{\rm (ii)}]
   For every $\beta>1$, there exists a constant $C_{\beta}>1$ such that
   \begin{equation*}\label{eq:NTAR-preserving}
      \Psi_p(\Gamma_{\beta}(p))\subset\big\{w\in\mathbb B^n\!: |w-\nu_p|<C_{\beta}(1-|w|)\big\},
   \end{equation*}
where $\Gamma_{\beta}(p)$ is as in \eqref{defn:nontangential}.
}
\end{enumerate}
\end{proposition}

\begin{proof}
Without loss of generality, we may assume that $\nu_p=e_1$. Then by Theorem \ref{thm:continuity}, we know that $\Psi_p\!: \overline{\Omega}\to \overline{\mathbb B}^n$ is a homeomorphism with $\Psi_p(p)=e_1$.

{\rm (i)} According to the definition of horospheres, it suffices to prove that
\begin{equation}\label{eq:horo-preserving}
\lim_{\Omega\ni w\to p}\big(k_{\Omega}(z, w)-k_{\Omega}(z_0, w)\big)
=\lim_{\mathbb B^n\ni w\to e_1}\big(k_{\mathbb B^n}(\Psi_p(z),\, w)-k_{\mathbb B^n}(\Psi_p(z_0),\,w)\big)
\end{equation}
for all $z$, $z_0\in\Omega$.

We use the notation introduced in Section \ref{subsect:spherical rep}, and first show that for every complex geodesics $\varphi$ of $\Omega$ with $\varphi(1)=p$, $\Psi_p\circ\varphi$ is a complex geodesic of $\mathbb B^n$ and
\begin{equation}\label{eq:der-preserving}
\langle (\Psi_p\circ\varphi)'(1), e_1\rangle=\langle \varphi'(1), e_1\rangle.
\end{equation}
Indeed, for every such $\varphi$, we can rewrite it as the composition  $\varphi=\varphi_v\circ\sigma$, where $\varphi_v$ is the preferred complex geodesic of $\Omega$ associated to $v:=\varphi'(1)/|\varphi'(1)|\in L_p$, and $\sigma$ is a suitable element of ${\rm{Aut}}(\Delta)$ with $\sigma(1)=1$. Now by the definition of $\Psi_p$, it follows that $\Psi_p\circ\varphi_v=\eta_v$. We then see that
   $$
   \Psi_p\circ\varphi=\eta_v\circ\sigma\in \mathcal{O}(\overline{\Delta})
   $$
is a complex geodesic of $\mathbb B^n$ and particularly $(\Psi_p\circ\varphi)'(1)$ makes sense. Moreover,
   $$
   \langle (\Psi_p\circ\varphi)'(1),\, e_1\rangle=\sigma'(1)\langle\eta'_v(1),\, e_1\rangle
   =\sigma'(1)\langle v, e_1\rangle^2=\sigma'(1)\langle\varphi'_v(1),\, e_1\rangle
   =\langle\varphi'(1),\, e_1\rangle
   $$
as desired.

Now fix a pair of distinct points $z$, $z_0\in\Omega$, and we come to prove equality  (\ref{eq:horo-preserving}). Let $\varphi$ be the unique complex geodesic of $\Omega$ such that $\varphi(0)=z_0$ and $\varphi(1)=p$. Then from  what we have proved it follows that the left-hand side of equality (\ref{eq:horo-preserving}) is equal to
\begin{equation}\label{eq:pf of horo-pres1}
\begin{split}
\lim_{\mathbb R\ni t\to 1^{-}}&\big(k_{\Omega}(z,\, \varphi(t))-k_{\Omega}(z_0,\,\varphi(t))\big)
=\lim_{\mathbb R\ni t\to 1^{-}}\big(k_{\Omega}(z,\, \varphi(t))-k_{\Delta}(0, t)\big)\\
=&\lim_{\mathbb R\ni t\to 1^{-}}\big(k_{\mathbb B^n}(\Psi_p(z),\,\Psi_p\circ\varphi(t))-
  k_{\mathbb B^n}(\Psi_p\circ\varphi(0),\, \Psi_p\circ\varphi(t))\big)\\
 &+\lim_{\mathbb R\ni t\to 1^{-}}\big(k_{\Omega}(z,\, \varphi(t))-
  k_{\mathbb B^n}(\Psi_p(z),\, \Psi_p\circ\varphi(t))\big)\\
=&\lim_{\mathbb B^n\ni w\to e_1}\big(k_{\mathbb B^n}(\Psi_p(z),\, w)-
  k_{\mathbb B^n}(\Psi_p(z_0),\, w)\big)\\
 &+\lim_{\mathbb R\ni t\to 1^{-}}\big(k_{\Omega}(z,\, \varphi(t))-k_{\mathbb B^n}(\Psi_p(z),\,
  \Psi_p\circ\varphi(t))\big).\\
\end{split}
\end{equation}
The proof will be complete by showing that
\begin{equation}\label{eq:key limit1}
\lim_{\mathbb R\ni t\to 1^{-}}
\big(k_{\Omega}(z,\, \varphi(t))-k_{\mathbb B^n}(\Psi_p(z),\,\Psi_p\circ\varphi(t))\big)=0.
\end{equation}

Let $\psi$ be the unique complex geodesic of $\Omega$ such that $\psi(0)=z$ and $\psi(1)=p$. Let $\rho\in \mathcal{O}(\Omega,\,\Omega)$ and $\varrho\in \mathcal{O}(\mathbb B^n,\, \mathbb B^n)$ be the Lempert projections associated to $\psi$ and $\Psi_p\circ\psi$, respectively. Note that in view of the Hopf lemma, the continuous curve $[0, 1)\ni t\mapsto \varphi(t)$ is non-tangential, it follows from Proposition \ref{prop:basic facts} (ii) that
   $$
   |k_{\Omega}(z,\,\varphi(t))-k_{\Omega}(z,\, \rho\circ\varphi(t))|
   \leq  k_{\Omega}(\varphi(t),\, \rho\circ\varphi(t))\to 0
   $$
as $t\to 1^-$. Similarly,
   $$
   \big|k_{\mathbb B^n}(\Psi_p(z),\, \Psi_p\circ\varphi(t))-
   k_{\mathbb B^n}(\Psi_p(z), \,\varrho\circ\Psi_p\circ\varphi(t))\big|
   \leq k_{\mathbb B^n}(\Psi_p\circ\varphi(t), \,\varrho\circ\Psi_p\circ\varphi(t))\to 0
   $$
as $t\to 1^-$. As a result, we see that equality (\ref{eq:key limit1}) is equivalent to
\begin{equation}\label{eq:key limit2}
\lim_{\mathbb R\ni t\to 1^{-}}\big(k_{\Omega}(z,\, \rho\circ\varphi(t))-
k_{\mathbb B^n}(\Psi_p(z), \,\varrho\circ\Psi_p\circ\varphi(t))\big)=0.
\end{equation}
Now using the explicit formula for $k_{\Delta}$ and Proposition \ref{prop:basic facts} (i), we deduce that
\begin{equation*}
\begin{split}
&\lim_{\mathbb R\ni t\to 1^{-}}\big(k_{\Omega}(z, \,\rho\circ\varphi(t))
 -k_{\mathbb B^n}(\Psi_p(z),\, \varrho\circ\Psi_p\circ\varphi(t))\big)\\
=&\,\lim_{\mathbb R\ni t\to 1^{-}}\big(k_{\Delta}(0, \,\psi^{-1}\circ\rho\circ\varphi(t))
 -k_{\Delta}(0, \,(\Psi_p\circ\psi)^{-1}\circ\varrho\circ\Psi_p\circ\varphi(t))\big)\\
=&\,\frac12\lim_{\mathbb R\ni t\to 1^{-}}\log
  \bigg(\frac{1-|(\Psi_p\circ\psi)^{-1}\circ\varrho\circ\Psi_p\circ\varphi(t)|}{1-t}
    \cdot\frac{1-t}{1-|\psi^{-1}\circ\rho\circ\varphi(t)|}
  \bigg)\\
=&\,\frac12\log
  \frac{d((\Psi_p\circ\psi)^{-1}\circ\varrho)_{e_1}((\Psi_p\circ\varphi)'(1))}
  {d(\psi^{-1}\circ\rho)_p(\varphi'(1))}\\
=&\,\frac12\log
  \bigg(\frac{\langle(\Psi_p\circ\varphi)'(1), \,e_1\rangle}{\langle(\Psi_p\circ\psi)'(1),\, e_1\rangle}
     \cdot\frac{\langle\psi'(1),\, e_1\rangle}{\langle\varphi'(1),\, e_1\rangle}
  \bigg),
\end{split}
\end{equation*}
which is equal to zero in view of equality (\ref{eq:der-preserving}). The penultimate equality follows from a simple geometrical consideration togther with the Hopf lemma, or alternatively from the classical Julia-Wolff-Cararth\'{e}odory theorem (see, e.g., \cite[Section 1.2.1]{Abatebook}, \cite[Chapter VI]{Sarasonbook}).

Now equality (\ref{eq:key limit2}) (and hence (\ref{eq:key limit1})) follows. The proof of (i) is complete.

\medskip
{\rm (ii)} For every $\beta>1$, we set
\begin{equation}\label{direction-cone}
\mathcal{V}_{\beta}:=\big\{v\in L_p\!: \varphi_v(\Delta)\cap\Gamma_{\beta}(p)\neq\emptyset\big\}.
\end{equation}
Then by using the continuity of $\Psi^{-1}_p$ and \cite[Theorem 2]{Huang-Pisa94}, we can argue as in the proof of \cite[Lemma 3.4]{Bracci-Trans} to conclude that for every $\beta>1$, $\mathcal{V}_{\beta}$ is relatively compact in $L_p$. Therefore,
\begin{equation}\label{centercpt1}
   \{\eta_v(0)\!: v\in \mathcal{V}_{\beta}\}\subset\subset\mathbb B^n
\end{equation}
and
\begin{equation}\label{centercpt2}
   \{\varphi_v(0)\!: v\in \mathcal{V}_{\beta}\}
   =\Psi^{-1}_p(\{\eta_v(0)\!: v\in \mathcal{V}_{\beta}\})
   \subset\subset\Omega.
\end{equation}
Together with the following simple estimate (see, e.g., \cite[Theorem 2.3.51]{Abatebook}):
   $$
   \sup_{(z,\, w)\in\Omega\times K}k_{\Omega}(z, w)+
   \frac12\log{\rm dist}(z, \partial\Omega)<\infty
   $$
for all compact sets $K\subset\Omega$, it follows that
\begin{equation}\label{uniform-est1}
   \sup_{(v,\, \zeta)\in \mathcal{V}_{\beta}\times \Delta}
   \frac{{\rm dist}(\varphi_v(\zeta), \partial\Omega)}{1-|\zeta|}<\infty.
\end{equation}
On the other hand, in light of Theorem \ref{thm:continuity} (i) we see that the function
  $$
  \frac{\varphi_v(\zeta)-p}{\zeta-1}
  =\int_{0}^{1}\varphi_v'(t\zeta+(1-t))dt
  $$
is continuous on $L_p\times\overline{\Delta}$, and nowhere vanishing there by the injectivity of $\varphi_v$ on $\overline{\Delta}$. Thus
\begin{equation*}
   \inf_{(v,\, \zeta)\in \mathcal{V}_{\beta}\times \overline{\Delta}}
   \bigg|\frac{\varphi_v(\zeta)-p}{\zeta-1}\bigg|>0,
\end{equation*}
which, combined with \eqref{direction-cone} and \eqref{uniform-est1}, implies that
  $$
  \bigcup_{v\in \mathcal{V}_{\beta}}\varphi_v^{-1}(\Gamma_{\beta}(p))
  \subset\big\{\zeta\in \Delta\!: |\zeta-1|<\widetilde{C}_{\beta}(1-|\zeta|)\big\}
  $$
for some sufficiently large $\widetilde{C}_{\beta}>1$. Now to complete the proof, it suffices to simply take
  $
  C_{\beta}:=2\widetilde{C}_{\beta}/\inf\limits_{v\in \mathcal{V}_{\beta}}\langle v, e_1\rangle.
  $
\end{proof}

\subsection{Complex Monge-Amp\`{e}re equations with boundary singularity}
Let $\Omega$ be a domain  in $\mathbb C^n\, (n>1)$ and denote by ${\rm Psh}(\Omega)$ the real cone of plurisubharmonic functions on $\Omega$. Then according to Bedford-Taylor \cite{Bedford82}, the complex Monge-Amp\`{e}re operator $(dd^c)^n$ (here $d^c=i(\bar{\partial}-\partial)$) can be defined for all $u\in {\rm Psh}(\Omega)\cap L^{\infty}_{{\rm loc}}(\Omega)$;  see alternatively \cite{Blockinote, Demaillynote, Demaillybook, Klimekbook, Kolodziej05, GZ-MAeqnbook}  for details. A very deep theorem of Bedford-Taylor \cite{Bedford76, Bedford82}, which is in many ways central to pluripotential theory, states that a function $u\in {\rm Psh}(\Omega)\cap  L^{\infty}_{{\rm loc}}(\Omega)$ solves the homogeneous complex Monge-Amp\`{e}re equation $(dd^c u)^n=0$ on $\Omega$ if and only if it is {\it maximal} on $\Omega$, in the sense of Sadullaev; namely, for every open set $G\subset\subset\Omega$ and every $v\in{\rm Psh}(G)$ satisfying that
   $$
   \limsup_{G\ni z\to x}v(z)\leq u(x)
   $$
for all $x\in\partial G$, it follows that $v\leq u$ on $G$.

We are now in a position to prove Theorem $\ref{mainthm2}$.

\begin{proof}[Proof of Theorem $\ref{mainthm2}$]
We first consider the special case when $\Omega=\mathbb B^n$. Set
   $$
   P_{\mathbb B^n,\, p}(z):=-\frac{1-|z|^2}{|1-\langle z, p\rangle|^2}.
   $$
Then $P_{\mathbb B^n,\, p}\in C^{\infty}(\overline{\mathbb B}^n\setminus\{p\})$. To prove that $P_{\mathbb B^n,\, p}$ is a solution to equation \eqref{MA-bp}, we need only verify that $P_{\mathbb B^n,\, p}$ is plurisubharmonic on $\mathbb B^n$ and $(dd^c P_{\mathbb B^n,\, p})^n$ vanishes identically there. Indeed, an easy calculation yields that
$$
\frac{\partial^2 P_{\mathbb B^n,\, p}}{\partial z_j\partial\overline z_k}(z)
=\frac{\delta_{j k}}{|1-\langle z, p\rangle|^2}+\frac{(1-\langle z, p\rangle)\overline{z}_jp_k+(1-\overline{\langle z, p\rangle})\overline{p}_jz_k-(1-|z|^2)\overline{p}_jp_k}{|1-\langle z, p\rangle|^4}
$$
for all $j,\, k=1,\ldots, n$, where $\delta_{j k}$ is the Kronecker delta. Note also that
$$2{\rm Re}(1-\langle z, p\rangle)-(1-|z|^2)=|z-p|^2,
$$
we then conclude that for every $z\in \mathbb B^n$ and $v\in \mathbb C^n$,
\begin{equation*}\label{hessian}
\begin{split}
\sum_{j,\, k=1}^n\frac{\partial^2 P_{\mathbb B^n,\, p}}{\partial z_j\partial\overline z_k}(z)v_j\overline{v}_k
=&\,\frac{|v|^2}{|1-\langle z, p\rangle|^2}+\frac{2{\rm Re}\big((1-\langle z, p\rangle)\langle v, z\rangle\langle p, v\rangle\big)-(1-|z|^2)|\langle p, v\rangle|^2}{|1-\langle z, p\rangle|^4}\\
=&\,|1-\langle z, p\rangle|^{-4}\Big(|1-\langle z, p\rangle|^2 |v|^2+|z-p|^2|\langle p, v\rangle|^2\\
&+2{\rm Re}\big((1-\langle z, p\rangle)\langle v, z-p\rangle\langle p, v\rangle\big)\Big)\\
=&\,|1-\langle z, p\rangle|^{-4}|(1-\langle z, p\rangle)v+\langle v, p\rangle (z-p)|^2,
\end{split}
\end{equation*}
which is obviously nonnegative and equal to $0$ if and only if $v=\lambda(z-p)$ with $\lambda\in\mathbb C$. This means that $P_{\mathbb B^n,\, p}\in{\rm Psh}(\mathbb B^n)$ and $(dd^c P_{\mathbb B^n,\, p})^n=0$ on $\mathbb B^n$, as desired. Moreover, by equality (\ref{eq:horosphere in ball}), it holds that
\begin{equation}\label{eq:horosphere-ball-sublevel1}
   E_{\mathbb B^n}(p, 0, R)=\big\{z\in\mathbb B^n\!: P_{\mathbb B^n,\, p}(z)<-1/R\big\}
\end{equation}
for all $R>0$. In other words, the sub-level sets of $P_{\mathbb B^n,\, p}$ are precisely horospheres of $\mathbb B^n$ with center $p$.

We now consider the general case and assume without loss of generality that $\nu_p=e_1$. Let $\Psi_p\!: \overline{\Omega}\to \overline{\mathbb B}^n$ be the boundary spherical representation given in Section \ref{subsect:spherical rep}, and set
  $$
  P_{\Omega,\, p}:=P_{\mathbb B^n,\, e_1}\circ\Psi_p.
  $$
Then Theorem \ref{thm:joint-continuity} implies that $P_{\Omega,\, p}\in C(\overline{\Omega}\setminus\{p\})$ and $P_{\Omega,\, p}=0$ on $\partial\Omega\setminus\{p\}$. Also by Proposition \ref{prop:horo-preserving} (i) and equality \eqref{eq:horosphere-ball-sublevel1}, we see that the sub-level sets of $P_{\Omega,\, p}$ are precisely horospheres of $\Omega$ with center $p$. Moreover, in light of the proof of \cite[Theorem 5.1]{Bracci-Trans} (with a slight modification)
  one has the following generalized Phragm\'en-Lindel\"of property
  for $P_{\Omega,\, p}$:
\begin{equation*}\label{extremality}
\begin{split}
   P_{\Omega,\, p}
   =\sup\Big\{&u\in {\rm Psh}(\Omega)\!: \limsup_{z\to x} u(z)\leq 0 \ \hbox{for all}
   \ x\in \partial \Omega\setminus\{p\},\\
   &\liminf_{t\to 1^{-}} |u(\gamma(t))(1-t)|
   \geq 2{\rm Re}\, \langle\gamma'(1), e_1\rangle^{-1} \ \hbox{for all}\ \gamma\in \Gamma_p
   \Big\},
\end{split}
\end{equation*}
where $\Gamma_p$ is the set of non-tangential $C^\infty$-curves $\gamma\!:[0,1]\to \Omega\cup\{p\}$ terminating at $p$ and with $\gamma([0,1))\subset\Omega$. Combining this with the (upper semi-) continuity of $P_{\Omega,\, p}$ on $\Omega$, we then see that  $P_{\Omega,\, p}\in {\rm Psh}(\Omega)$.

To show that $(dd^c P_{\Omega,\, p})^n=0$ on $\Omega$, we proceed as follows. For every $z\in\Omega$, we can find a $v\in L_p$ such that the associated preferred complex geodesic $\varphi_v$ passes through $z$, i.e., $z\in\varphi_v(\Delta)$. Then we see that
\begin{equation}\label{slicePoisson}
P_{\Omega,\, p}\circ\varphi_v=P_{\mathbb B^n,\, e_1}\circ\Psi_p\circ\varphi_v
=P_{\mathbb B^n,\, e_1}\circ\eta_v=-P/\langle v, e_1\rangle^2,
\end{equation}
where
\begin{equation}\label{Poissonkernel}
   P(\zeta):=\frac{1-|\zeta|^2}{|1-\zeta|^2}
\end{equation}
is the classical Poisson kernel on $\Delta$, which is obviously harmonic there. This leads us to conclude that $P_{\Omega,\, p}$ is  maximal on $\Omega$ by \cite[Proposition 5.1.4]{Braccinote}, and hence $(dd^c P_{\Omega,\, p})^n=0$ on $\Omega$, in view of Bedford-Taylor \cite{Bedford76, Bedford82}.

\smallskip
Now it remains to show that $P_{\Omega,\, p}(z)\approx -|z-p|^{-1}$ as $z\to p$ non-tangentially. First of all, by Proposition \ref{prop:horo-preserving} (ii) we can find for every $\beta>1$ a constant $C_{\beta}>1$ such that
   \begin{equation}\label{eq:behavior of BSE}
      \Psi_p(\Gamma_{\beta}(p))\subset\big\{w\in\mathbb B^n\!: |w-e_1|<C_{\beta}(1-|w|)\big\},
   \end{equation}
where $\Gamma_{\beta}(p)$ is as in \eqref{defn:nontangential}. Next, we write
   $$
   P_{\Omega,\, p}(z)|z-p|
   =P_{\mathbb B^n,\, e_1}\circ\Psi_p(z)|\Psi_p(z)-e_1|\cdot\frac{|z-p|}{|\Psi_p(z)-e_1|}
   $$
and
   $$
   P_{\mathbb B^n,\, e_1}\circ\Psi_p(z)|\Psi_p(z)-e_1|
   =-\frac{1-|\Psi_p(z)|^2}{|1-\langle \Psi_p(z), e_1\rangle|}
     \frac{|\Psi_p(z)-e_1|}{|1-\langle \Psi_p(z), e_1\rangle|}.
   $$
In view of (\ref{eq:behavior of BSE}), we see that for every $z\in\Gamma_{\beta}(p)$,
   $$
   \frac{1}{C_{\beta}}\leq\frac{1-|\Psi_p(z)|}{|\Psi_p(z)-e_1|}
   \leq\frac{1-|\Psi_p(z)|^2}{|1-\langle \Psi_p(z), e_1\rangle|}
   =(1+|\Psi_p(z)|)\frac{1-|\Psi_p(z)|}{|1-\langle \Psi_p(z), e_1\rangle|}
   \leq 2
   $$
and
   $$
   1\leq\frac{|\Psi_p(z)-e_1|}{|1-\langle \Psi_p(z), e_1\rangle|}
   \leq C_{\beta}\frac{1-|\Psi_p(z)|}{|1-\langle \Psi_p(z), e_1\rangle|}
   \leq C_{\beta}.
   $$

We are left to examine the behavior of the quotient $\frac{|z-p|}{|\Psi_p(z)-e_1|}$ as $z$ in $\Gamma_{\beta}(p)$. We follow an argument in \cite{Bracci-MathAnn}. First of all, it follows  from the definition of $\Gamma_{\beta}(p)$ and \eqref{eq:behavior of BSE} that
   $$
   |z-p|\approx {\rm dist}(z, \partial\Omega), \qquad
   |\Psi_p(z)-e_1|\approx {\rm dist}(\Psi_p(z), \partial\mathbb B^n)
   $$
for all $z\in\Gamma_{\beta}(p)$. Here the implicit constants depend only on $\beta$. On the other hand, by the well-known boundary estimates of the Kobayashi distance on bounded strongly pseudoconvex domains with $C^2$-boundary (see, e.g., \cite[Theorems 2.3.51 and 2.3.52]{Abatebook}), we know that
   $$
   \sup_{z\in\Omega}\big|k_{\Omega}(z, \Psi_p^{-1}(0))+
   \frac12\log{\rm dist}(z, \partial\Omega)\big|<\infty,
   $$
and the same is true for $\mathbb B^n$. Therefore, passing to the logarithm, it remains to show that for every $\beta>1$,
\begin{equation}\label{ineq:key estimate for Kob-dist1}
\sup_{z\in\Gamma_{\beta}(p)}\big|k_{\mathbb B^n}(\Psi_p(z), 0)-k_{\Omega}(z, \Psi_p^{-1}(0))\big|<\infty.
\end{equation}
To this end, we first conclude from \eqref{centercpt1} and \eqref{centercpt2} that
   $$
   C'_{\beta}:=\sup_{v\in \mathcal{V}_{\beta}}k_{\mathbb B^n}(\eta_v(0), 0)<\infty,
   $$
and
   $$
   C''_{\beta}:=\sup_{v\in \mathcal{V}_{\beta}}k_{\Omega}(\varphi_v(0), \Psi_p^{-1}(0))<\infty,
   $$
where $\mathcal{V}_{\beta}$ is as in \eqref{direction-cone}. Now for every $z\in\Gamma_{\beta}(p)$, let $v\in \mathcal{V}_{\beta}$ be such that
$z\in\varphi_v(\Delta)$. Then
\begin{equation*}\label{ineq:key estimate for Kob-dist2}
\begin{split}
   \big|k_{\Omega}(z, \Psi_p^{-1}(0))&-k_{\Delta}(\varphi_v^{-1}(z), 0)\big|
   =\big|k_{\Omega}(z, \Psi_p^{-1}(0))-k_{\Omega}(z, \varphi_v(0))\big|\\
   &\leq k_{\Omega}(\varphi_v(0), \Psi_p^{-1}(0))\leq C''_{\beta},
\end{split}
\end{equation*}
and
\begin{equation*}\label{ineq:key estimate for Kob-dist3}
\begin{split}
   \big|k_{\mathbb B^n}(\Psi_p(z), 0)&-k_{\Delta}(\varphi_v^{-1}(z), 0)\big|
   =\big|k_{\mathbb B^n}(\Psi_p(z), 0)-k_{\Delta}(\eta_v^{-1}\circ\Psi_p(z), 0)\big|\\
   &=\big|k_{\mathbb B^n}(\Psi_p(z), 0)-k_{\mathbb B^n}(\Psi_p(z), \eta_v(0))\big|\\
   &\leq k_{\mathbb B^n}(\eta_v(0), 0)\leq C'_{\beta}.
\end{split}
\end{equation*}
Combining these two estimates leads to
   $$
   \sup_{z\in\Gamma_{\beta}(p)}\big|k_{\mathbb B^n}(\Psi_p(z), 0)-k_{\Omega}(z, \Psi_p^{-1}(0))\big|
   \leq C'_{\beta}+C''_{\beta}<\infty,
   $$
and  \eqref{ineq:key estimate for Kob-dist1} follows.

The proof is now complete.
\end{proof}

Let $\Omega$, $p$ and $P_{\Omega,\, p}$ be as described in the above proof. Then the function $P_{\Omega}\!:(\overline{\Omega}\times\partial\Omega)\setminus{\rm diag }\, \partial\Omega\to (-\infty, 0]$ given by
       $$
       P_{\Omega}(z, p)=P_{\Omega,\, p}(z)
       $$
    is continuous, where
       $$
       {\rm diag }\, \partial\Omega:=\big\{(z, z)\in \mathbb C^{2n}\!: z\in\partial\Omega\big\}.
       $$
    This follows immediately from Theorem \ref{thm:joint-continuity} together
    with the fact that
    \begin{equation}\label{pluri-Poissonkernel}
       P_{\Omega}(z, p)=-\frac{1-|\Psi_p(z)|^2}{|1-\langle \Psi_p(z), \nu_p\rangle|^2}.
    \end{equation}

The following result concerns the uniqueness of solutions to equation \eqref{MA-bp}, which is a slight refinement of
\cite[Theorem 7.1]{Bracci-Trans}.

\begin{proposition}\label{prop:uniqueness-MA}
Let $\Omega \subset\mathbb C^n\, (n>1)$ be a bounded strongly linearly convex domain with $C^3$-smooth boundary. Let $p\in\partial\Omega$ and  $\nu_p$ be the unit outward normal to $\partial\Omega$ at $p$.  Then $P_{\Omega,\, p}$ is the unique solution to equation \eqref{MA-bp} with the additional property that
\begin{equation}\label{asymptotic limits}
\lim_{t\to 1^{-}} u\circ\gamma(t)(1-t)=-{\rm Re}\frac{2}{\langle\gamma'(1), \nu_p\rangle}
\end{equation}
for all $\gamma\in \Gamma_p$, the set of non-tangential $C^\infty$-curves $\gamma\!:[0,1]\to \Omega\cup\{p\}$ terminating at $p$ and with $\gamma([0,1))\subset\Omega$.
\end{proposition}

\begin{proof}
Arguing exactly as in the proof of \cite[Theorem 5.1]{Bracci-Trans}, we can see that $P_{\Omega,\, p}$ is indeed a solution to equation \eqref{MA-bp} with the descried property as in \eqref{asymptotic limits}. To show the uniqueness, suppose that $u$ is another such solution. Then by combining \cite[Proposition 11]{Lempert84} with the proof of \cite[Proposition 7.4]{Bracci-Trans}, we conclude that for every $v\in L_p$,
$u\circ\varphi_v$ is a negative harmonic function on $\Delta$ with
\[
\lim_{\zeta\to \xi}u\circ\varphi_v(\zeta)=0
\]
for all $\xi\in\partial\Delta\setminus\{1\}$, where $\varphi_v$ denotes the preferred complex geodesic of $\Omega$ associated to $v$ (with base point $p$). Thus from the classical Herglotz representation theorem it follows that $u\circ\varphi_v=c_vP$ for some constant $c_v<0$. Here as usual, $P$ is the classical Poisson kernel on $\Delta$. By \eqref{asymptotic limits}, we see that
\[
c_v=-{\rm Re}\frac{1}{\langle\varphi_v'(1), \nu_p\rangle}=-\frac{1}{\langle v, \nu_p\rangle^2}.
\]
Now combining this with \eqref{slicePoisson} (with $e_1$ replaced by $\nu_p$) yields that $u=P_{\Omega,\, p}$. This concludes the proof.
\end{proof}

We now conclude this paper by the following
\begin{remark}\label{finalrmk}
Let $\Omega \subset\mathbb C^n\, (n>1)$ be a bounded strongly linearly convex domain with $C^3$-smooth boundary.
\begin{enumerate}[leftmargin=2.0pc, parsep=4pt]
\item [{\rm (i)}]
   If further $\Omega$ is strongly convex  and $\partial\Omega$ is $C^{\infty}$-smooth, then for every $p\in\partial\Omega$ our solution $P_{\Omega,\, p}$ to equation \eqref{MA-bp} constructed as above coincides with $\widetilde{P}_{\Omega,\, p}$, the one by Bracci-Patrizio in \cite{Bracci-MathAnn} (using the boundary
spherical representation of Chang-Hu-Lee \cite{Chang-Hu-Lee88}, which is generally different from ours). To see this, one may use Proposition \ref{prop:uniqueness-MA} and \cite[Corollary 5.3]{Bracci-Trans}. Another more direct way goes as follows: By \eqref{slicePoisson},
  $$
  P_{\Omega,\, p}\circ\varphi_v=-P/\langle v, \nu_p\rangle^2
  $$
for all $v\in L_p$, where $\varphi_v$ is the preferred complex geodesic of $\Omega$ associated to $v$ and $P$ is as in \eqref{Poissonkernel}. Also, we have
   $$
   \widetilde{P}_{\Omega,\, p}\circ\widetilde{\varphi}_v=-P/\langle v, \nu_p\rangle^2
   $$
(see, e.g., \cite[equality (1.2)]{Bracci-Trans}), where $\widetilde{\varphi}_v$ is the unique complex geodesic of $\Omega$ such that $\widetilde{\varphi}_v(1)=p$, $\widetilde{\varphi}_v'(1)=\langle v, \nu_p\rangle v$ and ${\rm Im}\langle\widetilde{\varphi}_v''(1), \nu_p\rangle=0$. Now once noticing that by Theorem \ref{mainthm1} each $\widetilde{\varphi}_v$ coincides with $\varphi_v$ after composing a parabolic automorphism of $\Delta$ fixing $1$, under which $P$ is invariant, the desired result follows immediately.

\item  [{\rm (ii)}]
Analogous to \cite[Theorem 6.1]{Bracci-Trans}, we also have
   $$
   P_{\Omega}(z, p)=-\frac{\partial g_{\Omega}}{\partial \nu_p}(z, p),
   \quad  (z, p)\in\Omega\times\partial\Omega,
   $$
   where $g_{\Omega}$ is the pluricomplex Green function of $\Omega$ (see \eqref{Greenkernel}).
   This follows easily from two different ways of expressing the Busemann function of $\Omega$ at $p\in\partial\Omega$:
   $$
   B_{\Omega,\, p}(z, z_0):=\lim_{w\to p}\big(k_{\Omega}(z, w)-k_{\Omega}(z_0, w)\big),\quad (z, z_0)\in\Omega\times\Omega.
   $$
   Indeed, by \cite[Theorem 2.6.47]{Abatebook} (which is  also valid for the strongly linearly convex case, as we explained at the very beginning of Subsection \ref{props-BSR}) we have
   $$
   B_{\Omega,\, p}(z, z_0)=\frac12\log\bigg(\frac{\partial g_{\Omega}}{\partial \nu_p}(z_0, p)\Big/\frac{\partial g_{\Omega}}{\partial \nu_p}(z, p)\bigg).
   $$
   On the other hand,  combining \eqref{eq:horo-preserving} with \eqref{pluri-Poissonkernel} yields that
   $$
   B_{\Omega,\, p}(z, z_0)=B_{\mathbb B^n,\, \nu_p}\big(\Psi_p(z), \Psi_p(z_0)\big)
   =\frac12\log\bigg(\frac{P_{\Omega}(z_0, p)}{P_{\Omega}(z, p)}\bigg).
   $$
   We then conclude that there exists a constant $C>0$, depending only on $p\in\partial\Omega$, such that
   $$
   P_{\Omega}(z, p)=-C\frac{\partial g_{\Omega}}{\partial \nu_p}(z, p),\quad
   (z, p)\in\Omega\times\partial\Omega.$$
   Now evaluating both sides at $z=\varphi_{\nu_p}(0)$ gives $C=1$, as desired.

\item  [{\rm (iii)}]
Very recently, Poletsky \cite{Poletsky20} introduced a sort of pluripotential compactification for a class of so-called locally uniformly pluri-Greenien complex manifolds, which includes bounded domains in
$\mathbb C^n$. He also proved using results in \cite{Bracci-Trans} that the boundary of the pluripotential
compactification of a bounded strongly convex  domain $\Omega$ in $\mathbb C^n\, (n>1)$ with $C^{\infty}$-smooth boundary is homeomorphic to the Euclidean boundary $\partial\Omega$; see
\cite[Example 7.3]{Poletsky20} for details. We remark here that a similar argument using results in this paper shows that the same result is also true when  $\Omega$ is only strongly linearly convex with $C^3$-smooth boundary.
\end{enumerate}
\end{remark}

\section*{Acknowledgements}
Part of this work was done while both authors were visiting Huzhou University in part of the summers of 2017 and 2018. Both authors would like to thank this institute for its hospitality during their visit. Most of this work was also carried out while the second author was a postdoctor at the Institute of Mathematics, AMSS, Chinese Academy of Sciences. He would like to express his deep gratitude to his mentor, Professor Xiangyu Zhou, for constant supports and encouragements. He would also like to thank Professors F. Bracci,
L. Lempert, and E. A. Poletsky for patiently answering his questions during his reading of their related work. Special thanks also go to Professor L. Lempert for providing the second author with a copy of his very
limitedly accessible work \cite{Lempert84}. Last but not least, both authors thank the anonymous referee for his/her reading of this paper.


\begin{thebibliography}{\,  Roman}
\bibitem[Aba88]{Abatehoro} M. Abate, {\it Horospheres and iterates of holomorphic maps}, Math. Z. {\bf 198} (1988), 225--238.

\bibitem[Aba89]{Abatebook} M. Abate,  Iteration Theory of Holomorphic  Maps on Taut Manifolds, Mediterranean Press, Rende, Cosenza, 1989.

\bibitem[APS04]{Andersson04} M. Andersson, M. Passare and R. Sigurdsson, Complex Convexity and Analytic
    Functionals, Progress in Mathematics, 225, Birkh\"auser Verlag, Basel, 2004.

\bibitem[BER99]{Baouendibook} M. S. Baouendi, P. Ebenfelt and L. P. Rothschild, Real Submanifolds in Complex Space and Their Mappings, Princeton Mathematical Series, 47. Princeton University Press, Princeton, NJ, 1999.

\bibitem[BZZ06]{BZZ06} L. Baracco, D. Zaitsev and G. Zampieri, {\it A Burns-Krantz type theorem for domains with corners},  Math. Ann. {\bf 336} (2006), 491--504.

\bibitem[BT76]{Bedford76} E. Bedford and B. A. Taylor, {\it The Dirichlet problem for a complex Monge-Amp\`{e}re equation}, Invent. Math. {\bf 37} (1976), 1--44.

\bibitem[BT82]{Bedford82} E. Bedford and B. A. Taylor, {\it A new capacity for plurisubharmonic functions}, Acta Math. {\bf 149} (1982), 1--40.

\bibitem[Blo00]{Blocki00} Z. Blocki, {\it The $C^{1,\,1}$ regularity of the pluricomplex Green function}, Michigan Math. J. {\bf 47} (2000),  211--215.

\bibitem[Blo01]{Blocki01} Z. Blocki, {\it Regularity of the pluricomplex Green function
 with several poles},  Indiana Univ. Math. J. {\bf 50} (2001),  335--351.

\bibitem[Blo02]{Blockinote} Z. Blocki, The Complex Monge-Amp\`{e}re Operator in Pluripotential Theory, lecture notes, available at \url{http://gamma.im.uj.edu.pl/~blocki/}.

\bibitem[BKR20]{Bracci20} F. Bracci, D. Kraus and O. Roth, {\it A new Schwarz-Pick lemma at the boundary and rigidity of holomorphic maps}, arXiv:2003.02019.

\bibitem[BP05]{Bracci-MathAnn} F. Bracci and G. Patrizio, {\it Monge-Amp\`{e}re foliations with singularities at the boundary of strongly convex domains},  Math. Ann. {\bf 332} (2005), 499--522.

\bibitem[BPT09]{Bracci-Trans} F. Bracci,  G. Patrizio and S. Trapani, {\it The pluricomplex Poisson kernel for strongly convex domains},  Trans. Amer. Math. Soc. {\bf 361} (2009), 979--1005.

\bibitem[BST20]{Bracci-ST} F. Bracci, A. Saracco and S. Trapani, {\it The pluricomplex Poisson kernel for strongly pseudoconvex domains}, arXiv:2007.06270.

\bibitem[BT07]{Braccinote} F. Bracci and S. Trapani, {\it Notes on pluripotential theory}, Rend. Mat. Appl. {\bf 27} (2007), 197--264.

\bibitem[BK94]{Burns-Krantz} D. M. Burns and S. G. Krantz, {\it Rigidity of holomorphic mappings and a new Schwarz lemma at the boundary},  J. Amer. Math. Soc. {\bf 7} (1994), 661--676.

\bibitem[CHL88]{Chang-Hu-Lee88} C. H. Chang, M. C. Hu and H. P. Lee, {\it Extremal analytic discs with prescribed boundary data}, Trans. Amer. Math. Soc. {\bf 310} (1988), 355--369.

\bibitem[Chi83]{Chirka83} E. M. Chirka, {\it Regularity of the boundaries of analytic sets}, Math. USSR Sb. {\bf 45} (1983), 291--335.

\bibitem[CCS99]{Chirka-CS99} E. M. Chirka, B. Coupet and A. B. Sukhov, {\it On boundary regularity of analytic discs}, Michigan Math. J. {\bf 46} (1999),  271--279.

\bibitem[Dem91]{Demaillynote} J.-P. Demailly, Potential Theory in Several Complex Variables, available at \url{https://www-fourier.ujf-grenoble.fr/~demailly/manuscripts/trento2.pdf}.

\bibitem[Dem12]{Demaillybook} J.-P. Demailly, Complex Analytic and Differential Geometry, available  at \url{https://www-fourier.ujf-grenoble.fr/~demailly/books.html}.

\bibitem[FS87]{FornaessSCV} J. E. Fornaess and B. Stensones, Lectures on Counterexamples in Several Complex Variables, Princeton University Press, Princeton 1987.

\bibitem[GS13]{Gaussier-Seshadri} H. Gaussier and H. Seshadri, {\it Totally geodesic discs in strongly convex domains},  Math. Z. {\bf 274} (2013),  185--197.

\bibitem[Gua98]{GuanMAeqn} B. Guan, {\it The Dirichlet problem for complex Monge-Amp\`{e}re equations and regularity of the pluri-complex Green function},  Comm. Anal. Geom. {\bf 6} (1998),  687--703;  A correction, {\bf 8} (2000), 213--218.

\bibitem[GZ17]{GZ-MAeqnbook} V. Guedj and A. Zeriahi, Degenerate Complex Monge-Amp\`{e}re Equations, EMS Tracts in Mathematics, 26. European Mathematical Society (EMS), Z\"{u}rich, 2017.

\bibitem[H\"or94]{HormanderConvexity} L. H\"ormander, Notions of Convexity, Progress in Mathematics, 127, Birkh\"auser, Boston, 1994.

\bibitem[Hua94a]{Huang-Illinois94} X. Huang, {\it A preservation principle of extremal mappings near a strongly pseudoconvex point and its applications}, Illinois J. Math. {\bf 38} (1994), 283--302.

\bibitem[Hua94b]{Huang-Pisa94} X. Huang, {\it A non-degeneracy property of extremal mappings and iterates of holomorphic self-mappings},  Ann. Scuola Norm. Sup. Pisa Cl. Sci.  {\bf 21} (1994), 399--419.

\bibitem[Hua95]{Huang-Canad95} X. Huang, {\it A boundary rigidity problem for holomorphic mappings on some weakly pseudoconvex domains},  Canad. J. Math. {\bf 47} (1995),  405--420.

\bibitem[JP13]{Jarnicki-Pflugbook} M. Jarnicki and P. Pflug, Invariant Distances and Metrics in Complex Analysis, De Gruyter Expos. Math. vol. 9, De Gruyter, Berlin, 2013.

\bibitem[Kli91]{Klimekbook} M. Klimek, Pluripotential Theory, London Math. Soc. Monographs, Oxford University Press, 1991.

\bibitem[Kob98]{Kobayashibook} S. Kobayashi,  Hyperbolic Complex Spaces, Springer, Berlin, 1998.

\bibitem[Kol05]{Kolodziej05} S. Kolodziej, {\it The complex Monge-Amp\`{e}re equation and pluripotential theory},  Mem. Amer. Math. Soc. {\bf 178} (2005), 64 pp.

\bibitem[KW13]{KW13}L. Kosi\'nski and T. Warszawski, {\it Lempert theorem for strongly linearly convex domains}, Ann. Polon. Math. {\bf 107} (2013), 167--216.

\bibitem[Lem81]{Lempert81}  L. Lempert, {\it La m\'{e}trique de Kobayashi et la repr\'{e}sentation des domaines sur la boule},  Bull. Soc. Math. France  {\bf 109} (1981), 427--474.

\bibitem[Lem83]{Lempert83} L. Lempert, {\it Solving the degenerate complex Monge-Amp\`{e}re equation with one concentrated singularity}, Math. Ann. {\bf 263} (1983), 515--532.

\bibitem[Lem84]{Lempert84}  L. Lempert, {\it Intrinsic distances and holomorphic retracts}, Complex Analysis and Applications '81 (Varna, 1981), 341--364, Publ. House Bulgar. Acad. Sci., Sofia, 1984.

\bibitem[Lem86]{Lempert86} L. Lempert, {\it A precise result on the boundary regularity of biholomorphic mappings}, Math. Z. {\bf 193} (1986), 559--579; Erratum, {\bf 206} (1991), 501--504.

\bibitem[PZ12]{Pflug-Zwonek}  P. Pflug and W. Zwonek, {\it Exhausting domains of the symmetrized bidisc},  Ark. Mat. {\bf 50} (2012), 397--402.

\bibitem[Pol83]{Poletsky83} E. A. Poletsky, {\it The Euler-Lagrange equations for extremal holomorphic mappings of the unit disk}, Michigan Math. J. {\bf 30} (1983), 317--333.

\bibitem [Pol20]{Poletsky20} E. A. Poletsky, {\it (Pluri)Potential Compactifications}, Potential Anal. {\bf 53} (2020), 231--245.

\bibitem[Sar94]{Sarasonbook} D. Sarason, Sub-Hardy Hilbert Spaces in the Unit Disk, University of Arkansas Lecture Notes in the Mathematical Sciences, vol. 10, John Wiley \& Sons Inc., New York, 1994.

\bibitem[Sho08]{Shoikhet} D. Shoikhet, {\it Another look at the Burns-Krantz theorem},  J. Anal. Math. {\bf 105} (2008), 19--42.

\bibitem[Ves81]{Vesentini81} E. Vesentini, {\it Complex geodesics}, Compositio Math. {\bf 44} (1981), 375--394.

\bibitem[Yau78]{Yau} S. T. Yau, {\it On the Ricci curvature of a compact K\"{a}hler manifold and the complex Monge-Amp\`{e}re equation. I}, Comm. Pure Appl. Math. {\bf 31} (1978), 339--411.

\bibitem[Zim18]{Zimmer18} A. Zimmer, {\it Two boundary rigidity results for holomorphic maps}, arXiv:1810.05669v2.

\end{thebibliography}
\end{document}